\documentclass{article}

\pdfoutput=1

\title{Packing measure of the linear Gauss system}
\author{Rafal Tryniecki }
\date{}

\usepackage[numbers]{natbib}
\usepackage[T1]{fontenc}
\usepackage[english]{babel}
\usepackage[utf8]{inputenc}
\usepackage{lmodern}
\usepackage{amsfonts}
\usepackage{amsmath}
\usepackage{graphicx}
\usepackage{changepage}
\usepackage{amsthm}
\usepackage{tikz}
\newtheorem{innercustomthm}{Theorem}
\newenvironment{customthm}[1]
  {\renewcommand\theinnercustomthm{#1}\innercustomthm}
  {\endinnercustomthm}

\newtheorem{innercustomlem}{Lemma}
\newenvironment{customlem}[1]
  {\renewcommand\theinnercustomlem{#1}\innercustomlem}
  {\endinnercustomlem}

\newtheorem{theorem}{Theorem}[section]
\newtheorem{defi}[theorem]{Definition}
\newtheorem{notation}[theorem]{Notation}
\newtheorem{prop}[theorem]{Proposition}
\newtheorem{lemma}[theorem]{Lemma}
\newtheorem{corollary}[theorem]{Corollary}
\newtheorem{observation}[theorem]{Observation}
\newtheorem*{theorem*}{Theorem}

\begin{document}

\maketitle
\begin{abstract}
For every $k \in \mathbb{N}$ let  $f_k:[\frac{1}{k+1}, \frac{1}{k}] \to [0,1]$ be decreasing, linear functions such that $f_k(\frac{1}{k+1}) = 1$ and $f_k(\frac{1}{k}) = 0$, $k = 1, 2, \dots$. We define iterated function system (IFS) $S_n$ by limiting the collection of functions $f_k$ to first n, meaning $S_n = \{f_k \}_{k=1}^n$. Let $J_n$ denote the limit set of $S_n$. Then $\lim\limits_{n\to \infty} \mathcal{P}_{h_n}(J_n) = 2$, where $h_n$ is the packing dimension of $J_n$ and $\mathcal{P}_{h_n}$ is the corresponding packing measure.
\end{abstract}
\section{Introduction}
\par Let $g_k(x) = \rho_k x + b_k$,  \mbox{($k = 1, 2, \dots, m$)} be a collection of linear contractions defined on the interval $[0,1]$ and such that $g_k([0,1])\subset [0,1]$. Assume additionally that this collection satisfies Open Set Condition (defined in Definition \ref{def: open set condition}). It is well known that the packing dimension of the limit set (defined as in Definition \ref{not: limit set})  is equal to $\alpha$, the unique positive solution of the implicit equation $\sum\limits_{k = 1}^m \rho_k^\alpha = 1$; the $\alpha$–dimensional packing measure $\mathcal{P}_\alpha (K)$ and the $\alpha$–dimensional Hausdorff measure $\mathcal{H}^\alpha (K)$, both are finite and positive.
In 1999 E. Ayer and R. S. Strichartz \cite{AyerStri} provided an algorithm for calculating the Hausdorff measure of the limit set of iterated function system consisting of maps $g_k$. In 2003 D. Feng \cite{Feng} provided similar formula for the packing measure in the same setup.
\par On the other hand, let $f_k:[\frac{1}{k+1},\frac{1}{k}] \to [0,1]$, $f_k(x) = \{\frac{1}{x}\}$ defines the well-known Gauss map. Then, let $g_k: [0,1] \to [\frac{1}{k+1}, \frac{1}{k}]$ be a collection of inverse maps $g_k = f_k^{-1}$. For each, we define and iterated function system $S_n$ (IFS) consisting of the maps $g_k$, $k  = 1, \dots , n$. Let $J_n$ be the Julia set (limit set) generated by $S_n$. The asymptotics of Hausdorff dimension of $J_n$ was estimated first in 1929 by V. Jarnik \cite{jarnik1}, and then more precisely in 1992 by Doug Hensley \cite{Hensley}. In 2016 Mariusz Urbański and Anna Zdunik in \cite{UZ} proved, using Hensley's result that for previously mentioned sets, we have continuity of the Hausdorff measure in Hausdorff dimension, meaning
\[
\lim \limits_{n \to \infty}H_{h_n}(J_n)  =  1
\]
where $H_h$ - denotes the numerical value of Hausdorff measure in dimension $h$.
In this paper, we combine these two ideas by considering linear, decreasing functions 
\[
f_k\left(\frac{1}{k+1}\right) = 1  \text{ and } f_k\left(\frac{1}{k}\right) = 0
\]
and their inverses $g_k = f_k^{-1}$. Then, we define and iterated function system $S_n$, consisting of the maps $g_k$, $k  = 1, \dots , n$.  Let $J_n$ be the limit set generated by $S_n$. We prove that the packing measure is continuous, meaning
\begin{customthm}{\ref{thm: granica}}
Let $S_n$ be the iterated function system defined in Definition \ref{IFS Sn}. Then
\[
\lim \limits_{n \to \infty} \mathcal{P}_{h_n}(J_n) = 2
\]
where $J_n$ is the limit set of the IFS $S_n$ and $P_h$ denotes packing measure in packing dimension $h$.
\end{customthm}
The proof splits into two main parts: one is to estimate the packing measure from below, which is done in Section \ref{sec: from below}. This is the easier part of the proof. Then, in Section \ref{sec: from above} we estimate the lower limit of the $n$-th density (see Definition \ref{def: dn}) on different families of sets: starting with sets of form $[0,r]$, and slowly expanding families up to a point from which we can conclude, that the lower limit of $n$-th density on all subintervals of $[0,1]$ centered at $J_n$ is at least $\frac{1}{2}$.
\section{Notation and definitions}
Let $f_k(x):[\frac{1}{k+1}, \frac{1}{k}] \to [0,1]$ for $k \in \{1, 2, \dots \}$ be a linear linear analogue of the Gauss map. This means linear, decreasing function, such that 
\[
f_k\left(\frac{1}{k+1}\right) = 1  \text{ and } f_k\left(\frac{1}{k}\right) = 0
\]
\begin{defi}\label{IFS Sn}
Iterated function system (IFS) $S_n$ is defined by limiting the collection of functions $f_k$ to first $n$, meaning $S_n = \{f_k\}_{k=1}^{n}$. 
\end{defi}
By $g_k$ we will denote the inverse map $f_k^{-1}$. 
\begin{notation}\label{not: limit set}
By $J_n$ we will denote the limit set created by the IFS $S_n$. 
\[
J_n = \bigcap\limits_{l=1}^\infty \bigcup\limits_{q_1, q_2 \dots q_l \in \{1, 2, \dots  n \}^{l} } g_{q_1}\circ g_{q_2} \circ \dots \circ g_{q_l}([0,1])
\]
\end{notation}
\begin{defi}\label{defi: delta packing}
$\delta$-packing of a given set $E\subset\mathbb{R}^n$ is a countable family of pairwise disjoint closed balls of radii at most $\delta$ and with centers in $E$. 
\end{defi}
For $s \geq 0$, the $s$-dimensional packing premeasure of $E$ is defined as follows.
\begin{defi}\label{defi: packing premeasure}
\[
P^{premeasure}_s(E) = \inf\limits_{\delta > 0} \left \{ P_\delta^s  \right\}
\]
where $ P_\delta^s = \sup \{ \sum_{B_i \in \mathcal{R}}  |B_i|^{s}: \mathcal{R} \text{ is a } \delta-\text{packing of E} \}$ and $|B_i|$ denotes the diameter of $B_i$. 
\end{defi}
\begin{defi}\label{defi: packing measure}
The $s$-dimensional packing measure of E is defined as
\[
\mathcal{P}_s(E) = \inf \left \{ \sum\limits_{i=1}^{\infty}P^{premeasure}_s(E_i): E \subset \bigcup\limits_{i = 1}^{\infty} E_i \right\}
\]
\end{defi}
\begin{defi}\label{defi: packing dimension}
The packing dimension of E is by definition the quantity
\[
dim_P(E) := \inf \{s \geq 0: \mathcal{P}_{s}(E) = 0 \} = \sup \{s \geq 0: \mathcal{P}_{s}(E) = \infty \} 
\]
\end{defi}
\begin{notation}\label{hausdorff dimension}
We will denote Hausdorff dimension of the set $J_n$ by $h_n$ and Hausdorff measure of the set A in dimension $h$ by $H_{h}(A)$. 
\end{notation}
\begin{notation}\label{not: packing dimension}
Analogously, by $P_{h^p}^{A}$ we denote the packing measure of set A in the packing dimension $h^p$. By $h^p_n$ we denote the packing dimension of set $J_n$.
\end{notation}
\iffalse
\begin{defi}\label{defi: geometric measure}
Let $X$ be a Borel bounded subset of $\mathbb{R}^n$, $n \geq 1$. A borel probability measure $\mu$ on $X$ is said to be a geometric measure with an exponent $t \geq 0$ if and only if there exists a constant $C\geq 1$ such that
\[
C^{-1} \leq \frac{\mu(B(x,r))}{r^t} \leq C
\]
for every $x \in X$ and every $0<r \leq 1$
\end{defi}
We know, that $\mu(A) = m_n(A)$ is a geometric measure on $J_n$. From this fact and the following theorem, we get an equality of the Hausdorff and packing dimension.
\begin{theorem}\label{thm: rownosc wymiarow}
Let X be a Borel bound subset of $\mathbb{R}^n$, $n \geq 1$ and $\mu$ is a geometric measure on $X$ with an exponent $t$. Then
\[
HD(X) = PD(X) = t
\]
where $HD(X)$ denotes Hausdorff dimension and $PD(X)$ denotes packing dimension. 
\end{theorem}
The proof of this theorem can be found in \cite{przyturb} in chapter 8.6. From now on, because Hausdorff and packing dimensions are equal, we will be using $h_n$ to denote the packing and Hausdorff measure.
\fi
We denote by ${\rm diam}(F)$ the diameter of the set F. We will also use the notation $|F|$ to denote the diameter of the set $F$.
\begin{defi}\label{def: open set condition}
 We say that IFS composed of contractions $\{\phi_i\}_{i = 1}^n$ fulfills Open Set Condition (OSC), if there exists an open set V such that the following two conditions hold
\begin{equation}
    \bigcup\limits_{i=1}^n \phi_i(V) \subseteq V
\end{equation}
and the sets $\phi_i(V)$ are pairwise disjoint.
\end{defi}
Based on the fact that IFS $S_n$ fulfills the Open Set Condition, it is well known that packing dimension and Hausdorff dimensions are equal (see Theorem 2.7 of \cite{Falconerdrugi}) . Due to this fact, we will be using $h_n$ to denote the packing and Hausdorff dimension of the set $J_n$. We know that $h_n$ is a unique solution to the following equation 
\[
\sum\limits_{k=1}^{n} \left( \frac{1}{k} - \frac{1}{k+1}\right)^{h_n} = 1
\]
Proof of this fact can be found in \cite{Falconer}.
It follows from this equation that $\lim\limits_{n \to \infty} h_n = 1$ and $0<h_n<h_{n+1}<1$.
\begin{defi}\label{defi: miara unormowana}
Let $\mathcal{P}^{h_n}$ denote the packing measure in the packing dimension $h_n$. If $0<\mathcal{P}^{h_n}(X)<\infty$ then by $m_n$ we will denote the normalized $\mathcal{P}^{h_n}$ packing measure
\[
m_n(A) := \frac{\mathcal{P}^{h_n}(A \cap X)}{\mathcal{P}^{h_n}(X)}
\]
\end{defi}

Let A be a Borel subset of $[0,1]$. Then for all $k \in \{1,2, \dots, n \}$ 
\[
m_n(g_k(A)) = |g_k'|^{h_n} \cdot m_n(A)
\]
$m_n$ is thus the $h_n$-conformal measure for the set $J_n$. It is worth noting, that $m_n$ is a unique $h_n$-conformal measure on the set $J_n$. This implies, that the following holds true
\[
m_n(A) = \frac{H^{h_n}(A \cap J_n)}{H^{h_n}(J_n)}
\]
Using the fact that $m_n$ is conformal measure, we immediately get the following Lemma, which we will be using throughout our proof.
\begin{lemma}\label{lem: niezmienniczosc gestosci}
Let $A$ be Borel subset of $[0,1]$. Then for all $k \in \{1,2, \dots, n \}$ the following holds
\[
\frac{m_n(A)}{({\rm diam}A)^{h_n}} = \frac{m_n(g_k(A))}{({\rm diam}(g_k(A)))^{h_n}}
\]
\end{lemma}
We can develop this further.
\begin{lemma}\label{lem: rozciaganie krotkich przedzialow}
Let $B$ be Borel subset of $ \left(\frac{1}{k+1},\frac{1}{k}\right)$ for some $k \in \{1,2, \dots, n \}$ intersecting $J_n$. Then there exists a set $\hat{B}$ - a Borel subset of $[0,1]$ such that 
\[
\frac{m_n(B)}{({\rm diam}B)^{h_n}} = \frac{m_n(\hat{B})}{({\rm diam}\hat{B})^{h_n}}
\]
and $\frac{1}{j} \in \hat{B}$ for some $j \in \{1,2, \dots, n \}$.
\end{lemma}
\begin{proof}
If $B\subset \left[\frac{1}{k+1}, \frac{1}{k}\right]$ for some $k \leq n$, then $B = g_k(A)$ for some $A \subset[0,1]$. We can apply Lemma \ref{lem: niezmienniczosc gestosci} with $g_k(A) = B$, giving us
\[
\frac{m_n(B)}{({\rm diam}B)^{h_n}} = \frac{m_n(g_k(A))}{({\rm diam}(g_k(A)))^{h_n}} = \frac{m_n(A)}{({\rm diam}A)^{h_n}} 
\]
Now notice that $ 2 \cdot {\rm diam}B \leq {\rm diam}A$, and either there exists $j \in \{1,2, \dots, n \}$ such that $\frac{1}{j} \in A$ and then set $\hat{B} = A$ or we can apply this procedure again to the set A expanding it, until it intersects the set $ \{\frac{1}{j}\}_{j = 1}^{n}$, ending the proof. 
\end{proof}
Theorem 8.6.2 in \cite{przyturb}, called by the authors Frostman's type lemma states the following. Let $\mu$ be a Borel probability measure on $\mathbb{R}^n$ and A be a bounded subset of $\mathbb{R}^n$ then if there exists $C \in (0,\infty]$ such that 
\begin{enumerate}
    \item[(a)] for all $x \in A$ 
    \[
    \liminf\limits_{r \to 0} \frac{\mu(B(x,r))}{r^s} \leq C
    \]
    and
    \item[(b)]
    for all $x \in A$ 
    \[
    \liminf\limits_{r \to 0} \frac{\mu(B(x,r))}{r^s} \geq C^{-1} < \infty
    \]
\end{enumerate}
then $0 < \mathcal{P}_{s}(E) < \infty$ for each Borel set $E \subset A$.
\\
We know that $\mu = m_n$ fulfills this assumption. In fact, the following, well known Proposition holds.
\begin{theorem}\label{prop: ograniczenie unormowanej miary pakujacej z obu stron}
Let $m_n$ be the normalized packing measure on the set $J_n$. Then there exists constants $C_n > 0$ such that
\[
C_n^{-1} \cdot r^{h_n} \leq m_n(B(x,r)) \leq C_n \cdot r^{h_n}
\]
for each $x \in J_n$ and $r > 0$, and thus we know that $0 < \mathcal{P}_{h_n}(J_n) < \infty$.
\end{theorem}
\begin{proof}
This is a well known fact. We provide a proof for completeness. One can refer to the proof in \cite{Falconer} (Theorem 9.3). Let $J_n$ be the limit set of the IFS $S_n$. Let $\rho_j = |g_j'|$. The measure $m_n$ is distributed on the cylinders according to weights $\rho_j^{h_n}$ fulfilling
\[
\sum \limits_{j = 1}^n \rho_j^{h_n} = 1
\]
For $x \in J_n$ let $c_k(x)$ be the only interval corresponding to the cylinder of length $k$ in the symbolic space such that it contains $x$:  $c_k(x) = g_{i_k}\circ \dots \circ g_{i_1}([0,1]) \ni x$. Consider $B(x,r)$. Let $n_0 = min\{n: c_n(x) \subset B(x,r) \}$. Then $m_n(c_{n_0}(x)) \leq m_n(B(x,r))$, and $c_{n_0-1}(x) \not \subset B(x,r)$. Thus $|c_{n_0-1}(x)| \geq r$, and hence $|\rho_{i_1,\dots, i_{n_0-1}}| = |\rho_{i_1} \dots \rho_{i_{n_0-1}}| \geq r$. This implies that $|\rho_{i_1} \dots \rho_{i_{n_0}}| \geq r \cdot c$, where $c = \min\limits_{i \in \{1 \dots n\}}(\rho_i)$.
From this we get
\[
m_n(B(x,r)) \geq m_n(c_{n_0}(x)) = \rho_{i_1}^{h_n} \dots \rho_{i_{n_0}}^{h_n} \geq c^{h_n} \cdot r^{h_n} = \tilde{c} \cdot r^{h_n}
\]
which concludes the first part of the proof. For the second inequality, start by fixing $r>0$. For each sequence $\hat{i} = i_1, i_2, i_3 \dots $ we choose the least $k = k(\hat{i})$ such that 
\begin{equation}\label{eq: cylindry w druga strone}
\rho_{i_1} \cdot \dots \cdot \rho_{i_k} \leq r
\end{equation}
Then $\rho_{i_1} \cdot \dots \cdot \rho_{i_{k-1}} > r$, and thus 
\begin{equation}\label{eq: wielosc cylindrow}
\rho_{i_1} \cdot \dots \cdot \rho_{i_{k}} \geq r \cdot \min \limits_{1 \leq j \leq n} \rho_j = r \cdot c 
\end{equation}
Hence for any such $k = k(\hat{i})$ cylinder $c_{k(\hat{i})}(x)$ contains some ball of radius $r \cdot c$ and is contained in some ball of radius $r$. Moreover, each of the cylinders $c_{k(\hat{i})}$ are pairwise disjoint. Now take any $x \in J_n$ and $B(x,r)$. We cover the set $B(x,r)\cap J_n$ by cylinders $c = c_{n(\hat{i}}$. Those cylinders are from different generations, but due to inequality \eqref{eq: wielosc cylindrow}, they have almost the same size $r$. We denote this collection of cylinders by $\mathbf{C}$. Thus,
\[
m_n(B(x,\frac{r}{2})) \leq \sum \limits_{c \in \mathbf{C}: c \cap B(x,\frac{r}{2}) \neq \emptyset} m_n(c) \leq N \max m_n(c) \leq N \cdot C \cdot r^{h_n}
\]
In the first inequality we use the fact, that the ball $B(x,\frac{r}{2})$ can be intersected by at most $N = \frac{2}{c}+2$ disjoint cylinders from $\mathbf{C}$. For each of the cylinders from $\mathbf{C}$ we have $m_n(c) = \rho_{i_1}^{h_n} \dots \rho_{i_n}^{h_n} \leq r^{h_n}$ based on inequality \eqref{eq: cylindry w druga strone}, which ends the proof.
\end{proof} 

\begin{defi}\label{generacje fnl}
    Let $S_n$ be IFS generated by $f_k$, $k = 1, \dots, n$. We denote by $\mathcal{F}^{n}_l$ the $l$-th generation of intervals generated by $S_n$
\[
  \mathcal{F}^n_l = \left \{ g_{i_1}\circ g_{i_2}\circ \dots \circ g_{i_l}([0,1]): i_1, i_2, \dots, \in \{1,2,\dots, n\}  \right \}
\]

\end{defi}

\begin{defi}\label{def: dn}
  Let us denote $d_n(J) = \frac{m_n(J)}{(diam(J))^{h_n}}$. We call this value the $n$-th density of the interval J.  
\end{defi}

\begin{lemma}\label{lem: rozdzial gestosci na minimum}
Let $ [0,1] \supseteq I = I_1 \cup I_2$ where $I_1$ and $I_2$ are closed adjacent intervals. Then
\[
d_n\left(I \right) \geq \min \left \{ d_n(I_1), d_n(I_2) \right \}
\]
\end{lemma}
\begin{proof}
Note, that the measure $m_n$ is atomless. This is a direct consequence of the Proposition \ref{prop: ograniczenie unormowanej miary pakujacej z obu stron}. Using this fact together with $h_n<1$, we obtain
\[
d_n(J) = \frac{m_n(J)}{(diam(J))^{h_n}} = \frac{m_n(I_1 \cup I_2 )}{(diam(I_1 \cup I_2 ))^{h_n}} \geq \frac{m_n(I_1) + m_n(I_2)}{diam(I_1)^{h_n} + diam(I_2)^{h_n}} \geq 
\]
\[
\geq \min \left \{ \frac{m_n(I_1)}{diam(I_1)^{h_n}},\frac{m_n(I_2)}{diam(I_2)^{h_n}}\right \} = \min \left \{ d_n(I_1), d_n(I_2) \right \}
\]
\end{proof}
\begin{lemma}\label{lem: wypuklosc z hn}
Let $c$ be a real number with $0<c<1$, and $a_j \in (0,1)$ for $j \in 1, 2, \dots, k$. Then
\[
\sum \limits_{j = 1}^k a_j^c \geq \left [ \sum \limits_{j = 1}^k a_j\right]^c  
\]
\end{lemma}
\begin{proof}
\[
 1 = \sum \limits_{j = 1}^k \frac{a_j}{\sum \limits_{j = 1}^k a_j}  \leq \sum \limits_{j = 1}^k \Bigg(\frac{a_j}{\sum \limits_{j = 1}^k a_j} \Bigg)^c  = \sum \limits_{j = 1}^k \frac{a_j^c}{\left [ \sum \limits_{j = 1}^k a_j\right]^c}
\]
This inequality follows from the fact that $c<1$. Multiplying both sides by 
\[
\left [ \sum \limits_{j = 1}^k a_j\right]^c
\]
ends the proof.
\end{proof}
\section{Density theorems for the packing measure}
For a Borel measure in $\mathbb{R}^n$ let the lower density be defined as follows
\[
\Theta_*^{\alpha}(\mu,x) := \liminf\limits_{r \to 0} \frac{\mu(B(x,r))}{(2r)^\alpha}
\]
In particular, if $n = 1$, then
\[
\Theta_*^{\alpha}(\mu,x) := \liminf\limits_{r \to 0} \frac{\mu([x-r,x+r])}{(2r)^\alpha}
\]
P. Matilla proves in \cite{Matilla} (see 6.10) the following theorem.
\begin{theorem}\label{thm: gestosc miary pakujacej}
Suppose $A \subset \mathbb{R}^n$ with $\mathcal{P}_{\alpha}(A) < \infty$. Then
\[
\Theta_*^{\alpha}(\mathcal{P}_{\alpha}|_{A},x) = 1
\]
for $\mathcal{P}_{\alpha}$ almost all $x \in A$.
\end{theorem}
Based on this theorem, Feng in \cite{Feng} observed the following.
\begin{theorem}\label{thm: lower density feng}
Let $\mu$ be the normalized packing measure on the limit set $K \subset \mathbb{R}$ of the IFS consisting of finitely many linear, orientation preserving contractions, satisfying Open Set Condition. Then for $\mu$-almost all $x \in \mathbb{R}$, $\Theta_*^{\alpha}(\mu,x) = d_{min}$, where $d_{min}$ is defined as follows
\[
d_{min} = \inf\{d(J): J \text{ a closed interval centered in K with } J\subset [0,1] \}
\]
where $d(J) = \frac{\mu(J)}{|J|^{\alpha}}$ and $\alpha$ denotes packing dimension of K.
\end{theorem}
This theorem is easily adapted to the case, where iterated function system consists of finitely many linear, changing orientation contractions satisfying Open Set condition. An immediate consequence of this fact, we get the following.
Denote by $\mathcal{P}_{\alpha}|_K$ the restriction of the $\alpha$–dimensional packing measure on $K$, that is, $\mathcal{P}_{\alpha}|_K = \mathcal{P}_{h}(A \cap K)  $ for any Borel set $A \subset \mathbb{R}$. Since $\mu = c \cdot \mathcal{P}_{\alpha}|_K$ with $c = 1/\mathcal{P}_{\alpha}(K)$, we have
\begin{equation}\label{eq: lower density obs}
\Theta_*^{\alpha}(\mu,x) = \frac{1}{\mathcal{P}_{\alpha}(K)}\Theta_*^{\alpha}(\mathcal{P}_{\alpha}|_K,x)
\end{equation}
for all $x \in \mathbb{R}$.
With those tools in hand, Feng noticed the following, essential to his paper, theorem.
\begin{theorem}\label{thm: miara pakujaca jako odwrotnosc dmin}
\[
\mathcal{P}_{\alpha}(K) = d_{min}^{-1}
\]
\end{theorem}
\begin{proof}
From observation \eqref{eq: lower density obs}, we get $\Theta_*^{\alpha}(\mathcal{P}_{\alpha}|_K,x) = \Theta_*^{\alpha}(\mu,x)\mathcal{P}_{\alpha}(K)$ for any $x \in \mathbb{R}$. However, using the lower density Theorem \ref{thm: gestosc miary pakujacej}, we have that $\Theta_*^{\alpha}(\mathcal{P}_{\alpha}|_K,x) = 1$ for $\mathcal{P}_{\alpha}|_K$ almost all $x \in \mathbb{R}$ which implies the result.
\end{proof}
Adapting this theorem to our case, we get the following result.
\begin{theorem}\label{thm: tw o gestosci miary pakujacej}
\[
\mathcal{P}_{h_n}(J_n) = \sup \limits_{\substack{F \ centered \ at \ J_n \\ F \subseteq [0,1]}} \frac{(diam(F))^{h_n}}{m_n(F)} = \left(\inf \limits_{\substack{F \ centered \ at \ J_n \\ F \subseteq [0,1]}} d_n(F) \right)^{-1}
\] 
where $\mathcal{P}_{h_n}$ denotes packing measure with dimension $h_n$, $diam(F)$ is the diameter of the interval F, $m_n$ is the normalized packing measure and $d_n(F)$ denotes $n$-th density of the interval $F$.
\end{theorem}
\section{Estimating packing measure from below}\label{sec: from below}
In 2002 De-Jun Feng showed in \cite{Feng2} the following result regarding comparing the Hausdorff and packing measures on the real line.
\begin{theorem*}
For each $0 < s < 1$, define
\[
c(s) = \inf \limits_{E} \frac{\mathcal{P}_s(E)}{H_s(E)}
\]
where $\mathcal{P}_s$, $H_s$ denote the s-dimensional packing and Hausdorff measure respectively, and the infimum is taken over all sets $E \subset \mathbb{R}$ with $0<H_s(E)< \infty$. Then
\[
2^s(1+v(s))^s \leq c(s) \leq 2^s(2^{\frac{1}{s}}-1)^s
\]
for each $0< s < 1$ and $v(s) = min\{16^{-\frac{1}{1-s}}, 8^{-\frac{1}{(1-s)^2}} \}$.
\end{theorem*}
This result is a general estimate of the quotient of packing and Hausdorff measure. This result implies the following estimate.
\[
\liminf \limits_{n \to \infty} \mathcal{P}_{h_n}(J_n) \geq 2
\]
However, we are going to provide a short and direct proof of this fact for our case.
\begin{theorem}\label{thm: ograniczenie z dolu}
Let $S_n$ be iterated function system defined in (\ref{IFS Sn}). By $J_n$ denote the limit set of the IFS $S_n$. $\mathcal{P}_h$ denotes packing measure in packing dimension $h$ and $h_n$ the packing dimension of $J_n$. Then
\[
\liminf\limits_{n\to \infty} \mathcal{P}_{h_n}(J_n) \geq  2
\]
\end{theorem}
\begin{proof}

From definition of $g_k$, we know that
\[
g_1(x) = -\frac{1}{2}x + 1
\]
and 
\[
g_n(x) = \left ( \frac{1}{n+1} - \frac{1}{n} \right) x + \frac{1}{n}
\]
for $x \in [0,1]$. The leftmost point in $J_n$ is a stationary point such that $g_n\circ g_1(x_n) = x_n$. Short computation yields that
\[
x_n = \frac{2n}{2n^2+2n-1}
\]
Let us define an interval $I_n = [a_n,b_n]$ such that it is centered at $x_n$, and $b_n = \frac{1}{n}$. Then, because $x_n$ is the left endpoint of the set $J_n$, the following holds
\[
m_n([a_n,b_n]) = m_n\left(\left[\frac{1}{n+1}, \frac{1}{n}\right]\right) = \left|\frac{1}{n}-\frac{1}{n+1}\right|^{h_n}
\]
and 
\[
|b_n - a_n| = 2 \cdot | b_n - x_n| = 2\cdot \left( \frac{1}{n} - \frac{2n}{2n^2+2n-1} \right) = 2\cdot \frac{2n-1}{n\left(2n^2+2n-1\right)}
\]
Hence
\[
d_n(I_n) = \frac{m_n([a_n,b_n])}{|b_n-a_n|^{h_n}} = \frac{\left|\frac{1}{n}-\frac{1}{n+1}\right|^{h_n}}{ \left(2\cdot \frac{2n-1}{n\left(2n^2+2n-1\right)}\right)^{h_n}} = 
\left(\frac{1}{2}\right)^{h_n} \cdot \left ( \frac{2n^2+2n-1}{2n^2+n-1} \right)^{h_n}
\]
We showed that for each $J_n$, there is an interval $I_n$, centered at $J_n$ with density equal to $\left(\frac{1}{2}\right)^{h_n} \cdot \left ( \frac{2n^2+2n-1}{2n^2+n-1} \right)^{h_n}$, and thus based on Theorem \ref{thm: tw o gestosci miary pakujacej} we get
\[
 \mathcal{P}_{h_n}(J_n) \geq \left [\left(\frac{1}{2}\right)^{h_n} \cdot \left ( \frac{2n^2+2n-1}{2n^2+n-1} \right)^{h_n} \right]^{-1}
\]
which implies
\[
\liminf \limits_{n \to \infty}  \mathcal{P}_{h_n}(J_n) \geq 2
\]
\end{proof}
\section{Estimating packing measure from above}\label{sec: from above}
We will show that the function, which assigns to every number $n \in \mathbb{N}$ the packing measure of the $J_n$ in its packing dimension $h_n$ has the limit equal to $2$ when $n$ tends to infinity. We do this in several steps - first by showing that the lower limit of the densities of the intervals $[0,r]$ is at least 1/2, then expanding the family of intervals with this property up to a family of intervals that has an right endpoint in set $\frac{1}{k}$, $k = 0,1, \dots$ - in Lemmas \ref{k+l do k} - \ref{dlugie lewe}. Then, we deduce the same about the intervals with left endpoint in set $\frac{1}{k}$, $k = 0,1, \dots$ - Lemmas \ref{krotkie prawe} - \ref{dlugie prawe}. Then, in the final propositions we use previously obtained estimates to show that any interval centered at $J_n$ that intersects set $\frac{1}{k}$, $k = 0,1, \dots$ has the limit of the densities at least 1/2. From those estimates, we will be able to deduce that all intervals centered at $J_n$ have this property. 
\subsection{Density on intervals $[0,r]$ and $\left[\frac{1}{k}, \frac{1}{l}\right]$}
\begin{theorem}\label{thm: zero-r}
\[
   \liminf \limits_{n \to \infty} \inf \left \{d_n([0,r]): r \in [0,1],  \text{ and interval } [0,r] \text{ is centered at } J_n  \right \} \geq \frac{1}{2}
\]
\end{theorem}
\begin{proof}
Notice that there exists $k \in \{1,2,\dots n-1\}$ such that $\frac{1}{k+1}<r\leq \frac{1}{k}$. Such $k$ exists, because $[0,r]$ is centered at $J_n$. Moreover, the fact that $[0,r]$ is centered at $J_n$, implies that $r/2>\frac{1}{n+1}$ and thus $\frac{1}{k}>\frac{2}{n+1}$, giving $k < (n+1)/2$. Then, by Lemma \ref{lem: wypuklosc z hn}
\[
d_n([0,r]) = \frac{m_n([0,r])}{(diam([0,r]))^{h_n}} \geq \frac{\sum\limits_{j = k+1}^n (1/j - 1/(j+1))^{h_n}}{(1/k)^{h_n}} \geq \frac{|\frac{1}{k+1}- \frac{1}{n+1}|^{h_n}}{|\frac{1}{k}|^{h_n}} \geq
\]
\[ 
\geq \left[\frac{n-k}{n+1} \cdot \frac{k}{k+1}\right ]^{h_n}
\]
Because $k<\frac{n+1}{2}$, we can see that the minimum value of this expression is attained at $k = 1$ or $k = \frac{n+1}{2}$. Indeed
\[
\frac{\partial}{\partial k} \frac{n-k}{n+1} \cdot \frac{k}{k+1}  = \frac{n-k(k+2)}{(n+1)(k+1)^2}
\]
is equal to zero if and only if $n - k(k+2) = 0$, hence $k = \sqrt{n+1} - 1$ - at which this expression attains maximum. Thus, the minimum values are attained at the edge of domain of the function. The value of the expression for $k = 1$ is equal to 
\[
\left[\frac{n-1}{n+1} \cdot \frac{1}{2}\right ]^{h_n}
\]
and for $k = \frac{n+1}{2}$
\[
\left[\frac{n-\frac{n+1}{2}}{n+1} \cdot \frac{\frac{n+1}{2}}{\frac{n+1}{2}+1}\right ]^{h_n} 
= \left [ \frac{1}{2} - \frac{2}{n+3}\right ]^{h_n}
\]
Hence
\[
 \lim \limits_{n \to \infty} \inf \left \{d_n([0,r]): r \in [0,1],  \text{ and interval } [0,r] \text{ is centered at } J_n  \right \} \geq \frac{1}{2}
\]
\end{proof}
Note, that the Theorem \ref{thm: zero-r} requires only $r>\frac{2}{n+1}$.  Hence, we will be using the following Corollary in the next steps.
\begin{corollary}\label{col: zero - r uproszczony}
\[
   \liminf \limits_{n \to \infty} \inf \left \{d_n([0,r]): r \in \left[\frac{2}{n+1}, 1 \right ] \right \} \geq \frac{1}{2}
\]
\end{corollary}
Now, we move to showing similar property for intervals with endpoints in the set $\left\{ \frac{1}{k} \right\}$, $k = 1, 2, \dots n$. 
\begin{theorem}\label{k+l do k}
\begin{equation*}
\begin{split}    
   \lim \limits_{n \to \infty} \inf \Bigg\{ d_n\left(\left[\frac{1}{k+l},\frac{1}{k}\right]\right) &: k \in \{1,2, \dots n \}, 
    \\
   &  l+k \in \{k+1, k+2, \dots, n+1\} \Bigg\} \geq 1
\end{split}
\end{equation*}
\end{theorem}
\begin{proof}
Notice that using Lemma \ref{lem: wypuklosc z hn}, we get
\[
d_n\left(\left[\frac{1}{k+l},\frac{1}{k}\right]\right) = \frac{m_n\left(\left[\frac{1}{k+l},\frac{1}{k}\right]\right)}{\left|\left[\frac{1}{k+l},\frac{1}{k}\right]\right|^{h_n}} \geq \frac{\left|\sum \limits_{j=k}^{k+l-1} \frac{1}{j} - \frac{1}{j+1} \right |^{h_n}}{\left|\frac{1}{k}-\frac{1}{k+l}\right|^{h_n}} = \frac{\left|{\frac{1}{k} - \frac{1}{k+l}}\right|^{h_n}}{\left|\frac{1}{k}-\frac{1}{k+l}\right|^{h_n}} = 1
\]
Hence, obviously
\begin{equation*}
\begin{split}    
   \liminf \limits_{n \to \infty} \inf \Bigg\{ d_n\left(\left[\frac{1}{k+l},\frac{1}{k}\right]\right) &: k \in \{1,2, \dots n \}, 
    \\
   &  l+k \in \{k+1, k+2, \dots, n+1\} \Bigg\} \geq 1
\end{split}
\end{equation*}
\end{proof}
\subsection{Density on intervals with right endpoint in set $\frac{1}{k}$, $k\in \mathbb{N}$}
Now, we will prove that the lower limit of the densities of intervals contained in $[\frac{1}{k+1},\frac{1}{k}]$ for some $k = 1, 2, \dots$, and having the right endpoint equal to $\frac{1}{k}$ is at least $\frac{1}{2}$.
\begin{theorem}\label{krotkie lewe}
\begin{equation*}
\begin{split}
\liminf \limits_{n \to \infty} \inf  \Bigg\{d_n([r,\frac{1}{k}]) &: k \in \{1,2, \dots, n\},
\\
& r \in [\frac{1}{k+1}, \frac{1}{k}), \ interval \ centered \ at \ J_n \Bigg \} \geq \frac{1}{2}
\end{split}
\end{equation*}
\end{theorem}
\begin{proof}
Note that $[r,\frac{1}{k}] \subseteq [\frac{1}{k+1},\frac{1}{k}]$ and thus applying Lemma \ref{lem: niezmienniczosc gestosci} we get
\[
d_n ([r,\frac{1}{k}]) = d_n([0,\hat{r}])
\]
where $r = g_{k}(\hat{r})$.
Now we notice that centres of the intervals are transformed under $g_k$ to the centers of the intervals, and center of $[r,\frac{1}{k}]$ is in $J_n$. Thus, the interval $[0,\hat{r}]$ is also centered at $J_n$ and based on Theorem \ref{thm: zero-r} we get 
\begin{equation*}
\begin{split}
\liminf \limits_{n \to \infty} \inf  \Bigg\{d_n([r,\frac{1}{k}]) &: k \in \{1,2, \dots, n\},
\\
& r \in [\frac{1}{k+1}, \frac{1}{k}), \ interval \ centered \ at \ J_n \Bigg \} \geq \frac{1}{2}
\end{split}
\end{equation*}
\end{proof}
Similarly to Corollary \ref{col: zero - r uproszczony}, notice that the proof of this theorem only requires $r > |\frac{1}{k}-\frac{1}{k+1}| \cdot \frac{2}{n+1}$. This yields another corollary, used later on.
\begin{corollary}\label{cor: krotkie lewe bez centrowania}
\begin{equation*}
\begin{split}
\liminf \limits_{n \to \infty} \inf  \Bigg\{d_n([r,\frac{1}{k}]) &: k \in \{1,2, \dots, n\},
\\
& r \in [\frac{1}{k+1}, \frac{1}{k}), \\ & \left|\frac{1}{k} - r\right| > \left|\frac{1}{k}-\frac{1}{k+1}\right| \cdot \frac{2}{n+1} \Bigg \} \geq \frac{1}{2}
\end{split}
\end{equation*}
\end{corollary}
To expand this further, we will show that the lower limit of the densities of the intervals with right endpoint equal to $\frac{1}{k}$, $k \in \mathbb{N}$ and containing interval $[\frac{1}{k+1}, \frac{1}{k}]$ is at least 1/2.
\begin{theorem}\label{dlugie lewe}
\[
\liminf \limits_{n \to \infty} \inf \left \{d_n\left(\left[r,\frac{1}{k}\right]\right): k\in \{1, 2, \dots, n\}, r \in (0, 1/(k+1)), \text{centered at } J_n \right \} \geq \frac{1}{2}
\]
\end{theorem}
\begin{proof}
Let $l \in \mathbb{N}$ be such a number that $ \frac{1}{k+l+1}\leq r<\frac{1}{k+l}$. 
\\Part (A). First, assume that $k+l+1 \leq n+1$. Then, by Lemma \ref{lem: wypuklosc z hn}
\[
d_n\left(\left[r,\frac{1}{k}\right ] \right) = \frac{m_n\left(\left[r,\frac{1}{k}\right]\right)}{\left|r - \frac{1}{k}\right|^{h_n}} \geq \frac{m_n\left(\left[\frac{1}{k+l},\frac{1}{k}\right]\right)}{\left|\frac{1}{k+l+1} - \frac{1}{k}\right|^{h_n}} = \frac{\sum \limits_{j = k}^{k+l-1} \left|\frac{1}{j} - \frac{1}{j+1}\right|^{h_n}}{\left|\frac{1}{k+l+1} - \frac{1}{k}\right|^{h_n}} \geq
\]
\[
\geq \frac{\left|\sum \limits_{j = k}^{k+l-1} \frac{1}{j} - \frac{1}{j+1}\right|^{h_n}}{\left|\frac{1}{k+l+1} - \frac{1}{k}\right|^{h_n}} =  \frac{\left|\frac{1}{k+l} - \frac{1}{k}\right|^{h_n}}{\left|\frac{1}{k+l+1} - \frac{1}{k}\right|^{h_n}} = \left[\frac{l}{k(k+l)} \cdot \frac{k(k+l+1)}{l+1} \right ]^{h_n} =
\]
\[
= \left [\frac{l}{l+1} \cdot \frac{k+l+1}{k+l}\right]^{h_n} \geq \left(\frac{1}{2}\right)^{h_n}
\]
Note that in Part (A) of this theorem, we are not using the assumption that the interval $[r,\frac{1}{k}]$ is centered. This assumption is used later on, in Part (B).\\
\\Part (B). Now, assume that $r< \frac{1}{n+1}$. From the fact that the interval $[r,\frac{1}{k}]$ is centered at $J_n$, we know that the point $c$ - the center of  $[r,\frac{1}{k}]$ must be located to the right of the point $\frac{1}{n+1}$. Using this observation and Lemma \ref{lem: wypuklosc z hn} we get
\[ 
 d_n\left(\left[r,\frac{1}{k}\right ] \right) = \frac{m_n\left(\left[r,\frac{1}{k}\right]\right)}{\left|r - \frac{1}{k}\right|^{h_n}} = \frac{m_n\left(\left[\frac{1}{n+1},\frac{1}{k}\right]\right)}{\left|r - \frac{1}{k}\right|^{h_n}} = \frac{\sum \limits_{j = k}^{n} \left|\frac{1}{j} - \frac{1}{j+1}\right|^{h_n}}{\left|r - \frac{1}{k}\right|^{h_n}} \geq
\]
\[
\geq \frac{\left|\frac{1}{k} - \frac{1}{n+1}\right|^{h_n}}{\left|r - \frac{1}{k}\right|^{h_n}} \geq \frac{\left|\frac{1}{k} - c \right|^{h_n}}{\left|r - \frac{1}{k}\right|^{h_n}} = \left(\frac{1}{2}\right)^{h_n}
\]
Thus we get 
\[
\liminf \limits_{n \to \infty} \inf \left \{d_n([r,\frac{1}{k}]): k\in \{1,2,\dots, n\}, r \in (0, 1/(k+1)) \right \} \geq \frac{1}{2}
\]
\end{proof}
\subsection{Density on intervals with left endpoint in set $\frac{1}{k}$, $k = 0,1, \dots$}
Now we move to the case where we analyze the intervals with left endpoint in $\{\frac{1}{k+1}\}$, for some $k \leq n$, and contained in $[\frac{1}{k+1}, \frac{1}{k}]$
\begin{theorem}\label{krotkie prawe}
\[
\begin{split}
\liminf \limits_{n \to \infty} \inf \Bigg \{d_n([\frac{1}{k+1},r]) &: k \in \{1,2, \dots, n\},
\\ & r \in (\frac{1}{k+1}, \frac{1}{k}] \text{, centered at } J_n \Bigg \} \geq \frac{1}{2}
\end{split}
\]
\end{theorem}
\begin{proof}
Notice that the interval $[\frac{1}{k+1},r]$ is contained in $[\frac{1}{k+1},\frac{1}{k}]$. Thus
\[
d_n([\frac{1}{k+1},r]) = d_n(f_{k}([\frac{1}{k+1},r])) = d_n([\hat{r},1]) 
\]
for $\hat{r} = f_{k}(r)$. If $|\hat{r} - 1| > \frac{1}{2}$, then invoking Theorem \ref{dlugie lewe} with $k = 1$, we get our thesis.
If $|\hat{r} - 1| \leq \frac{1}{2}$, then we can apply the map $f_1$ to the interval $[\hat{r}, 1]$. Further, notice that because interval $[\hat{r},  1]$ was centered at $J_n$, then so must be the interval $f_1([\hat{r},  1]) = [0, \hat{\hat{r}}]$. Moreover, the interval $[\hat{r},  1]$ and $[0, \hat{\hat{r}}]$ has the same density based on Lemma \ref{lem: niezmienniczosc gestosci}  and so we can apply Theorem \ref{thm: zero-r}, which ends the proof.
\end{proof}
Notice that similarly to the proof of the Theorem \ref{thm: zero-r}, the only assumption needed in the proof of the Theorem \ref{krotkie prawe} is that $\left|\left[\frac{1}{k+1},r\right] \right| > \left|  \frac{1}{k} - \frac{1}{k+1}\right| \cdot \frac{1}{2} \cdot \frac{2}{n+1}$. This is due to the fact that in the first part of the proof, where we assume that $|\hat{r} - 1| > \frac{1}{2}$, we can use the Part (A) of the Theorem \ref{dlugie lewe}, which does not assume being centered at $J_n$. As for the other case, when $|\hat{r} - 1| < \frac{1}{2}$, we can invoke Corollary \ref{col: zero - r uproszczony} to end the proof. This observation gives another corollary.
\begin{corollary}\label{cor: krotkie prawe bez centrowania}
\[
\begin{split}
\liminf \limits_{n \to \infty} \inf \Bigg \{d_n\left(\left[\frac{1}{k+1},r\right]\right) &: k \in \{1,2, \dots, n\}, r \in \left(\frac{1}{k+1}, \frac{1}{k}\right], 
\\ & \left|\left[\frac{1}{k+1},r\right] \right| > \left|  \frac{1}{k} - \frac{1}{k+1}\right| \cdot \frac{1}{2} \cdot \frac{2}{n+1} \Bigg \} \geq \frac{1}{2}
\end{split}
\]
\end{corollary}
\begin{theorem}\label{dlugie prawe}
\[
\begin{split}    
\liminf \limits_{n \to \infty} \inf \Bigg \{d_n\left(\left[\frac{1}{k+l+1}, r \right] \right) &: k,l \in \mathbb{N}, k+l<n+1, l>0, 
\\  & r \in (1/(k+l),1] \text{ and } 1/(k+1) < r\leq 1/k \Bigg \} \geq \frac{1}{2}
\end{split}
\]
\end{theorem}
\begin{proof}
As a side note, one can notice that this case is symmetric to one in Theorem \ref{dlugie lewe}, although the method used in the proof is different.
First, let us assume that $|\frac{1}{k+1}-r|\geq \left|\frac{1}{k+1} - \frac{1}{k}\right| \cdot \frac{1}{2} \cdot \frac{2}{n+1}$. Additionally we can split the interval $\left[ \frac{1}{k+l+1}, r \right]$ into two - $[\frac{1}{k+l+1}, \frac{1}{k+1}]$ and $[\frac{1}{k+1}, r]$. Using Lemma \ref{lem: rozdzial gestosci na minimum}, we get 
\[
d_n\left(\left[\frac{1}{k+l+1}, r\right]\right) \geq \min \left \{d_n\left(\left[\frac{1}{k+l+1}, \frac{1}{k+1}\right]\right), d_n\left(\left[\frac{1}{k+1}, r \right]\right) \right\} \geq
\]
\[
\geq \min \left \{1, d_n\left(\left[\frac{1}{k+1}, r \right]\right) \right \}
\]
and using Theorem \ref{k+l do k} and Corollary \ref{cor: krotkie prawe bez centrowania} gives us the result.
Now if 
\[
\left|\frac{1}{k+1}-r \right |\leq \left|\frac{1}{k} - \frac{1}{k+1}\right| \cdot \frac{1}{2}\cdot \frac{2}{n+1}
\]
then, by Lemma \ref{lem: wypuklosc z hn}
\[
d_n\left(\left[\frac{1}{k+l+1}, r\right]\right)  \geq \frac{\sum\limits_{j = k+1}^{k+l}\left| \frac{1}{j} - \frac{1}{j+1}\right|^{h_n}}{\left| \frac{1}{k+l+1} - \frac{1}{k+1} + \left|\frac{1}{k} - \frac{1}{k+1}\right| \cdot \frac{1}{2}\cdot \frac{2}{n+1} \right|^{h_n}} \geq 
\]
\[
\geq \frac{\left|\sum\limits_{j = k+1}^{k+l} \frac{1}{j} - \frac{1}{j+1}\right|^{h_n}}{\left| \frac{1}{k+l+1} - \frac{1}{k+1}  + \left|\frac{1}{k} - \frac{1}{k+1}\right| \cdot \frac{1}{2}\cdot \frac{2}{n+1} \right|^{h_n}} =
\]
\[
= \frac{\left|\frac{1}{k+l+1} - \frac{1}{k+1}\right|^{h_n}}{\left| \frac{1}{k+l+1} - \frac{1}{k+1} + \left|\frac{1}{k+1} - \frac{1}{k}\right| \cdot \frac{1}{2}\cdot \frac{2}{n+1}\right |^{h_n}} =
\]
\[
= \frac{1}{\left| 1 + \frac{1}{n+1}\frac{\left|\frac{1}{k+1} - \frac{1}{k}\right|}{\left | \frac{1}{k+l+1} - \frac{1}{k+1}\right|}\right |^{h_n}} \geq
\]
\[
\geq \left (\frac{1}{2} \right)^{h_n}
\]
for sufficiently large n and all $k,l \in \mathbb{N}$ such that $k+l<n+1$.
\end{proof}
\subsection{Density on intervals intersecting the set $\frac{1}{k}$, $k = 0,1, \dots$}
Now we move to the final step of the proof. This step is split into three parts. Next two propositions require two auxillary lemmas, which we will formulate here and prove later on in Subsection \ref{subsec: dowody dodatkowych lematow}
\begin{customlem}{\ref{lem: jednostajnosc bledu 1}}
\[
\liminf \limits_{n\to\infty} \left( \inf\limits_{d > \frac{1}{n}} \left\{\frac{m_n\left(\left[\frac{1}{n+1},d \right]\right)}{\left|d-\frac{1}{n+1} \right |^{h_n}} \right \} \right) \geq 1 
\]
\end{customlem}
It is worth noting this theorem cannot be replaced by Theorem \ref{dlugie prawe}. This is due to the fact, that in the Theorem \ref{dlugie prawe} we have the estimate of the lower limit of the densities by $\frac{1}{2}$, whereas here we need a stronger estimate. 
\begin{customlem}{\ref{lem: jednostajnosc bledu 2}}
\[
\liminf \limits_{n\to\infty} \left( \inf\limits_{d < \frac{1}{2}} \left\{\frac{m_n\left(\left[d, 1 \right]\right)}{\left|1 - d \right |^{h_n}} \right \} \right) \geq 1 
\]
\end{customlem}
It is worth noting this theorem cannot be replaced by Theorem \ref{dlugie lewe}. In Theorem \ref{dlugie lewe} we have the estimate of the lower limit of the densities by $\frac{1}{2}$ and the intervls are required to be centered, whereas here we need a stronger estimate on all intervals.
\\
Now, let us start with the first proposition.
\begin{prop}\label{prop: a do b z jednym w srodku}
\[
\begin{split}  
\liminf \limits_{n \to \infty} \inf \Bigg \{d_n([a,b])&: 0 < a < \frac{1}{l+1} < \frac{1}{l} < b \leq 1, 
\\ &l \in \{2,3, \dots, n \} ,\ [a,b] \text{ centered at } J_n \Bigg \} \geq \frac{1}{2}
\end{split}
\]
\end{prop}
\begin{proof}
Let $k$ and $l$ be unique integer such that $[a,b] = [a,\frac{1}{k}]\cup[\frac{1}{k},\frac{1}{l}]\cup[\frac{1}{l},b]$ for some integer $l \leq k$ and $\frac{1}{k+1} \leq a < \frac{1}{k} \leq \frac{1}{l} \leq b < \frac{1}{l-1} $. First, assume that $k \geq n+1$. Note that this implies that $a \not \in J_n$. Let $c = \frac{a+b}{2}$ be the center of the interval $[a,b]$. Because the interval $[a,b]$ is centered at $J_n$, $c \geq \frac{1}{n+1}$, and thus $\frac{|\frac{1}{n+1}-a|}{|b-a|} < \frac{1}{2}$, which directly implies that $\frac{|b-\frac{1}{n+1}|}{|b-a|} \geq \frac{1}{2}$. Now, fix $\varepsilon > 0$. Using Lemma \ref{lem: jednostajnosc bledu 1} with $d=b$, fix $n$ large enough such that 
\[
\frac{m_n\left(\left[\frac{1}{n+1},b\right]\right)}{\left|b-\frac{1}{n+1}\right|^{h_n}} \geq (1 - \frac{\varepsilon}{2})
\]
for all $b \geq \frac{1}{n}$. Hence
\[
d_n\left(\left[a,b\right]\right) = \frac{m_n\left(\left[a,b\right]\right)}{\left|b-a\right|^{h_n}} = \frac{m_n\left(\left[\frac{1}{n+1},b\right]\right)}{\left|b-a\right|^{h_n}} = \frac{m_n\left(\left[\frac{1}{n+1},b\right]\right)}{\left|b-\frac{1}{n+1}\right|^{h_n}} \cdot \frac{\left|b-\frac{1}{n+1}\right|^{h_n}}{\left|b-a\right|^{h_n}} \geq 
\]
\[
\geq (1-\frac{\varepsilon}{2}) \cdot \left(\frac{1}{2}\right)^{h_n} \geq \frac{1}{2} - \varepsilon
\]
where last inequality comes from the fact that $h_n \to 1$ when $n \to \infty$ for all $n$ large enough.
\par From now on, we will assume that $k \leq n $. Let $J_1 = [a,\frac{1}{k}]$, $J_2 = [\frac{1}{k},\frac{1}{l}]$ and $J_3 = [\frac{1}{l},b]$. 
\par We say that $J_1$ is long when
\begin{equation}\label{eq: j1 long}
|J_1| = \left |a -\frac{1}{k}\right| \geq \left | \frac{1}{k} - \frac{1}{k+1} \right | \cdot \frac{2}{n+1}
\end{equation}
otherwise we say that $J_1$ is short. We say that $J_3$ is long when
\begin{equation}\label{eq: j3 long}
|J_3| = \left |\frac{1}{l} - b \right| \geq \left | \frac{1}{l-1} - \frac{1}{l} \right | \cdot \frac{1}{2} \cdot \frac{2}{n+1}
\end{equation}
otherwise we say it is short.
We will split this proof into four parts - based on the length of the intervals $J_1$ and $J_3$. Those cases are as follows
\begin{enumerate}
\item Both intervals $J_1$ and $J_3$ are long
\label{case 1}
\item  $J_1$ is long and $J_3$ is short
 \label{case 2}
\item $J_1$ is short and $J_3$ is long
\label{case 3}
\item  Both $J_1$ and $J_3$ are short
\label{case 4}
\end{enumerate} 
First, assume both $J_1$ and $J_3$ are long.
Then, by using Lemma \ref{lem: rozdzial gestosci na minimum}, we notice that
\[
d_n\left(\left[a,b\right]\right) = d_n\left(J_1\cup J_2 \cup J_3\right) \geq \min \left\{ d_n\left(J_1\right), d_n\left(J_2\right), d_n\left(J_3\right) \right\}   
\]
And thus using Theorem \ref{k+l do k}, Corollary \ref{cor: krotkie lewe bez centrowania} and \ref{cor: krotkie prawe bez centrowania} we get our thesis.
\par Now, moving to the second case, we assume that $J_1$ is long and $J_3$ is short.
Now, we split our interval into two $J_1$ and $J_2 \cup J_3$, out of which the first one is in form of the intervals from Corollary \ref{cor: krotkie lewe bez centrowania} and the other one is in form of ones from Theorem \ref{dlugie prawe}. Utilizing both of those theorems, and Lemma \ref{lem: rozdzial gestosci na minimum} we get
\[
d_n([a,b]) \geq \min \left \{ d_n(J_1), d_n(J_2 \cup J_3) \right \} \geq \frac{1}{2}
\]
In the third case, when $J_1$ is short and $J_3$ is long, we can split our interval into two $J_1 \cup J_2$ and $J_3$. Then the $J_1\cup J_2$ fulfills the assumptions of the Theorem \ref{dlugie lewe} Part (A), whereas $J_3$ fulfills the assumptions of the Corollary \ref{cor: krotkie prawe bez centrowania}. This together with Lemma \ref{lem: rozdzial gestosci na minimum} gives us
\[
d_n([a,b]) \geq \min \left \{ d_n(J_1 \cup J_2), d_n(J_3)\right \} \geq \frac{1}{2}
\]
In Case \ref{case 4}, assume that both of the intervals $J_1$ and $J_3$ are short i.e. do not satisfy \eqref{eq: j1 long} and \eqref{eq: j3 long}. By Lemma \ref{lem: wypuklosc z hn}
\[
d_n([a,b]) \geq \frac{\sum\limits_{j = l}^{k-1}\left | \frac{1}{j} - \frac{1}{j+1}  \right |^{h_n}}{\left | \left | \frac{1}{k} - \frac{1}{k+1} \right | \cdot \frac{2}{n+1} + \sum\limits_{j = l}^{k-1} \left |\frac{1}{j} - \frac{1}{j+1} \right|+ \left | \frac{1}{l-1} - \frac{1}{l} \right | \cdot \frac{1}{2} \cdot \frac{2}{n+1} \right  |^{h_n}} \geq
\]
\[
\geq \frac{\left | \sum\limits_{j = l}^{k-1} \frac{1}{j} - \frac{1}{j+1}  \right |^{h_n}}{\left | \left | \frac{1}{k} - \frac{1}{k+1} \right | \cdot \frac{2}{n+1} + \sum\limits_{j = l}^{k-1} \left |\frac{1}{j} - \frac{1}{j+1} \right| + \left | \frac{1}{l-1} - \frac{1}{l} \right | \cdot \frac{1}{2} \cdot \frac{2}{n+1} \right  |^{h_n}} \geq 
\]
\[
\geq \frac{\left | \frac{1}{l} - \frac{1}{k}  \right |^{h_n}}{\left | \frac{1}{l}-\frac{1}{k} + \frac{2}{n+1} \cdot \left ( \frac{1}{k(k+1)} + \frac{1}{2} \cdot \frac{1}{l(l-1)}\right )\right  |^{h_n}} \geq  \frac{1}{\left|1 + \frac{2}{n+1} \cdot \frac{\left ( \frac{1}{k(k+1)} + \frac{1}{2} \cdot \frac{1}{l(l-1)}\right )}{|\frac{1}{l}-\frac{1}{k}|} \right|^{h_n}} \geq \frac{1}{2}
\]
for sufficiently large n and all $l \in \{2,3 \dots, n\} $ and $k \in \{ l+1, l+2, \dots, n+1\}$.
\\
This ends the proof of the Proposition \ref{prop: a do b z jednym w srodku}.
\end{proof}
Now, we focus on the case, where there is no whole interval of the first generation in the interval $[a,b]$. 
\begin{prop}\label{prop: a do b z punktem w srodku}
\[
\begin{split}  
\liminf \limits_{n \to \infty} \inf \Bigg \{d_n([a,b]):& \frac{1}{k+1}< a < \frac{1}{k} < b <\frac{1}{k-1} \leq 1, 
\\ &k \in \{2,3, \dots, n + 1 \} ,\ [a,b] \text{ centered at } J_n \Bigg \} \geq \frac{1}{2}
\end{split}
\]
\end{prop}
\begin{proof}
Assume first that $k\leq n$. Case when $k = n+1$ will be analysed at the end of the proof. Let $J_1 = \left[a,\frac{1}{k}\right]$ and $J_2 = \left[\frac{1}{k},b\right]$. As in the proof of the Proposition \ref{prop: a do b z jednym w srodku}, we say that $J_1$ is long when
\begin{equation}\label{eq: j1 long 2}
|J_1| = \left |a -\frac{1}{k}\right| \geq \left | \frac{1}{k} - \frac{1}{k+1} \right | \cdot \frac{2}{n+1}
\end{equation}
and we call it short otherwise. We say that $J_2$ is long when
\begin{equation}\label{eq: j2 long 2}
|J_2| =  \left |\frac{1}{k} - b \right| \geq \left | \frac{1}{k-1} - \frac{1}{k} \right | \cdot \frac{1}{2} \cdot \frac{2}{n+1}
\end{equation}
It is worth noting, that 
\[
\frac{1}{k} - \left | \frac{1}{k} - \frac{1}{k+1} \right | \cdot \frac{1}{n+1} = g_k\left(\frac{1}{n+1}\right)
\] and 
\[
\frac{1}{k} + \left | \frac{1}{k-1} - \frac{1}{k} \right |\cdot \frac{1}{2} \cdot \frac{1}{n+1} = g_{k-1} \circ g_1 \left ( \frac{1}{n+1} \right )
\] The immediate consequence of this fact is the following observation. 
\begin{observation}\label{obs: j1 short then centered in j2}
Let $c$ be the center of the interval $[a,b]$. Assume that $[a,b]$ is centered at $J_n$. Moreover, assume that $J_1$ is short. Then $c$ must be to the right of the $\frac{1}{k}$. 
\end{observation}
\begin{proof}
Note by $c_1$ the center of the interval $J_1$. Because
\[
\frac{1}{k} - \left | \frac{1}{k} - \frac{1}{k+1} \right | \cdot \frac{1}{n+1} = g_k(\frac{1}{n+1})
\]
(see Figure \ref{fig:visualisation of jn}) and $J_1$ is short (i.e. does not fulfill \eqref{eq: j1 long 2}), then the center of the interval $J_1$ must be to the right of $g_k(\frac{1}{n+1})$. This means that $g_k(\frac{1}{n+1}) < c_1 $. But the center of the interval $[a,b]$ lies to the right of $c_1$ i.e. $c_1<c$, thus $g_k(\frac{1}{n+1}) < c$. Now, notice that $J_n\cap [g_k(\frac{1}{n+1}), \frac{1}{k}] = \emptyset$. This, together with $g_k(\frac{1}{n+1}) < c$ implies that $\frac{1}{k} < c$.
\end{proof} 
Now, we will formulate a symmetric observation regarding $J_2$. 
\begin{observation}\label{obs: j2 short then centered in j1}
Let $c$ be the center of the interval $[a,b]$. Assume that $[a,b]$ is centered at $J_n$. Moreover, assume that $J_2$ is short. Then $c$ must be to the left of the $\frac{1}{k}$. 
\end{observation}
\begin{proof}
Note by $c_2$ the center of the interval $J_2$. Because $\frac{1}{k} + \left | \frac{1}{k-1} - \frac{1}{k} \right | \cdot \frac{1}{2} \cdot \frac{1}{n+1} = g_{k-1} \circ g_1 \left ( \frac{1}{n+1} \right )$ (see Figure \ref{fig:visualisation of jn}) and $J_2$ is short (i.e. does not fulfill \eqref{eq: j2 long 2}), then the center of the interval $J_2$ must be to the left of $g_{k-1} \circ g_1 \left ( \frac{1}{n+1} \right )$. This means that $c_2<g_{k-1} \circ g_1 \left ( \frac{1}{n+1} \right ) $. But the center of the interval $[a,b]$ lies to the left of $c_2$ i.e. $c<c_2$, thus $c < g_{k-1} \circ g_1 \left ( \frac{1}{n+1} \right ) $. Now, notice that $J_n\cap \left[\frac{1}{k}, g_{k-1} \circ g_1 \left ( \frac{1}{n+1} \right) \right] = \emptyset$. This, together with $c < g_{k-1} \circ g_1 \left ( \frac{1}{n+1} \right ) $ implies that $c < \frac{1}{k} $.
\end{proof}
Now, we will split the proof of the Proposition \ref{prop: a do b z punktem w srodku} into four cases based on the length of intervals $J_1$ and $J_2$. 
\\Case 1. Assume that $J_1$ and $J_2$ are long.
Then utilising Lemma \ref{lem: rozdzial gestosci na minimum} yields
\[
d_n([a,b])\geq \min \left \{ d_n(J_1), d_n(J_2) \right \} \geq \frac{1}{2}
\]
based on Theorem \ref{krotkie lewe} and Corollary \ref{cor: krotkie prawe bez centrowania}. 

\begin{figure}

\tikzset{every picture/.style={line width=0.75pt}} %set default line width to 0.75pt        

\begin{tikzpicture}[x=0.75pt,y=0.75pt,yscale=-1,xscale=1]
%uncomment if require: \path (0,167); %set diagram left start at 0, and has height of 167

%Straight Lines [id:da32724783963761794] 
\draw [color={rgb, 255:red, 208; green, 2; blue, 27 }  ,draw opacity=1 ]   (77,82) -- (259,82) ;
%Straight Lines [id:da921309054462744] 
\draw    (334,70.5) -- (334,91.5) ;
%Curve Lines [id:da8811745961514654] 
\draw [color={rgb, 255:red, 208; green, 2; blue, 27 }  ,draw opacity=1 ]   (77,82) .. controls (102,47.5) and (153,55.5) .. (168,82) ;
%Curve Lines [id:da479745405829314] 
\draw [color={rgb, 255:red, 208; green, 2; blue, 27 }  ,draw opacity=1 ]   (168,82) .. controls (184,63.5) and (195,67.5) .. (206,82.5) ;
%Curve Lines [id:da8390806260623473] 
\draw [color={rgb, 255:red, 208; green, 2; blue, 27 }  ,draw opacity=1 ]   (206,82.5) .. controls (222,64) and (228,67.5) .. (239,82.5) ;
%Curve Lines [id:da42345130209271264] 
\draw [color={rgb, 255:red, 208; green, 2; blue, 27 }  ,draw opacity=1 ]   (239,82.5) .. controls (249,71.5) and (253,72.5) .. (259,82) ;
%Straight Lines [id:da4205588709390047] 
\draw [color={rgb, 255:red, 74; green, 144; blue, 226 }  ,draw opacity=1 ]   (425,82) -- (516,82) ;
%Curve Lines [id:da5331737997318009] 
\draw [color={rgb, 255:red, 74; green, 144; blue, 226 }  ,draw opacity=1 ]   (334,82) .. controls (359,47.5) and (410,55.5) .. (425,82) ;
%Curve Lines [id:da3199023838147188] 
\draw [color={rgb, 255:red, 74; green, 144; blue, 226 }  ,draw opacity=1 ]   (425,82) .. controls (441,63.5) and (452,67.5) .. (463,82.5) ;
%Curve Lines [id:da6446189268587634] 
\draw [color={rgb, 255:red, 74; green, 144; blue, 226 }  ,draw opacity=1 ]   (463,82.5) .. controls (479,64) and (485,67.5) .. (496,82.5) ;
%Curve Lines [id:da0039330673773332325] 
\draw [color={rgb, 255:red, 74; green, 144; blue, 226 }  ,draw opacity=1 ]   (496,82.5) .. controls (506,71.5) and (510,72.5) .. (516,82) ;
%Straight Lines [id:da3582069100701528] 
\draw [color={rgb, 255:red, 65; green, 117; blue, 5 }  ,draw opacity=1 ]   (426,82.31) -- (369,82.31) ;
%Curve Lines [id:da2809217202818659] 
\draw [color={rgb, 255:red, 65; green, 117; blue, 5 }  ,draw opacity=1 ]   (426,82.31) .. controls (418.17,69.08) and (402.2,72.15) .. (397.5,82.31) ;
%Curve Lines [id:da6680340428670739] 
\draw [color={rgb, 255:red, 65; green, 117; blue, 5 }  ,draw opacity=1 ]   (397.5,82.31) .. controls (392.49,75.21) and (389.04,76.75) .. (385.6,82.5) ;
%Curve Lines [id:da3564077279657185] 
\draw [color={rgb, 255:red, 65; green, 117; blue, 5 }  ,draw opacity=1 ]   (385.6,82.5) .. controls (380.59,75.41) and (378.71,76.75) .. (375.26,82.5) ;
%Curve Lines [id:da8596157041437231] 
\draw [color={rgb, 255:red, 65; green, 117; blue, 5 }  ,draw opacity=1 ]   (375.26,82.5) .. controls (372.13,78.28) and (370.88,78.67) .. (369,82.31) ;
%Straight Lines [id:da15248767785341744] 
\draw [color={rgb, 255:red, 65; green, 117; blue, 5 }  ,draw opacity=1 ]   (369,70.5) -- (369,91.5) ;
%Straight Lines [id:da9821580854752792] 
\draw [color={rgb, 255:red, 74; green, 144; blue, 226 }  ,draw opacity=1 ]   (516,70.5) -- (516,91.5) ;
%Straight Lines [id:da45505931014824097] 
\draw [color={rgb, 255:red, 208; green, 2; blue, 27 }  ,draw opacity=1 ]   (259,70.5) -- (259,91.5) ;
%Straight Lines [id:da3035038193871311] 
\draw    (77,70.5) -- (77,91.5) ;

% Text Node
\draw (335.5,105) node   [align=left] {\begin{minipage}[lt]{12.92pt}\setlength\topsep{0pt}
$\frac{1}{k}$
\end{minipage}};
% Text Node
\draw (261,105) node   [color={rgb, 255:red, 208; green, 2; blue, 27 }  ,opacity=1 ]  [align=left] {\begin{minipage}[lt]{16.32pt}\setlength\topsep{0pt}
$g_k(\frac{1}{n+1})$
\end{minipage}};
% Text Node
\draw (82,105) node   [align=left] {\begin{minipage}[lt]{16.32pt}\setlength\topsep{0pt}
$\frac{1}{k+1}$
\end{minipage}};
% Text Node
\draw (517,105) node   [align=left] {\begin{minipage}[lt]{16.32pt}\setlength\topsep{0pt}
$\frac{1}{k-1}$
\end{minipage}};
% Text Node
\draw (372,107) node   [color={rgb, 255:red, 65; green, 117; blue, 5 }  ,opacity=1 ] [align=left] {\begin{minipage}[lt]{16.32pt}\setlength\topsep{0pt}
$g_{k-1} \circ $
\end{minipage}};
% Text Node
\draw (410,105) node    [color={rgb, 255:red, 65; green, 117; blue, 5 }  ,opacity=1 ] [align=left] {\begin{minipage}[lt]{16.32pt}\setlength\topsep{0pt}
$ g_1(\frac{1}{n+1})$
\end{minipage}};

\end{tikzpicture}
\caption{Visualisation of $J_n$.}
\label{fig:visualisation of jn}
\end{figure}
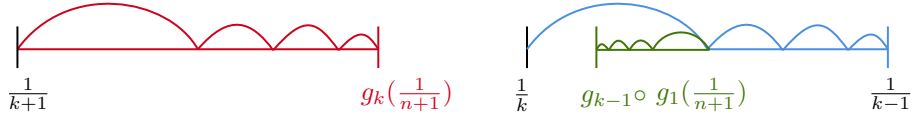
Case 2. Now, assume that $J_1$ is long and $J_2$ is short (ie. $J_1$ fulfills \eqref{eq: j1 long 2} and $J_2$ does not fulfill \eqref{eq: j2 long 2}). 
Fix $\varepsilon>0$. Let $c = center([a,b]) = (a+b)/2$ be the center of the interval $[a,b]$. $[a,b]$ is centered ad $J_n$. By the Observation \ref{obs: j2 short then centered in j1} we have $\frac{1}{k}< c$. Moreover, $c$ cannot be in the interval $[g_k(\frac{1}{n+1}), \frac{1}{k}]$, because the interval $[g_k(\frac{1}{n+1}),\frac{1}{k}]$ does not intersect $J_n$. This implies that $|[a,c]| \leq |[a,g_k(\frac{1}{n+1})]|$. Let $I_1 = [a,c]$ and $I_2 = [c,b]$. Because $c$ is the middle point of $[a,b]$, $|I_1| = |I_2|$.  From Lemma \ref{lem: jednostajnosc bledu 1} we can find $n$ to be large enough for the following to hold
\[
\frac{m_n\left(\left[\frac{1}{n+1},d\right]\right)}{\left|d-\frac{1}{n+1}\right|^{h_n}} \geq (1 - \frac{\varepsilon}{2})
\]
for all $d\geq \frac{1}{n}$. Now, applying Lemma \ref{lem: niezmienniczosc gestosci} to the interval $[a, g_k(\frac{1}{n+1})]$, we get
\begin{equation}\label{eq: jedn bledu dla j1 dlugiego}
\frac{m_n([a,g_k(\frac{1}{n+1})])}{|g_k(\frac{1}{n+1})- a|^{h_n}} \geq (1 - \frac{\varepsilon}{2})
\end{equation}
Using this fact and $|[a,c]| \leq |[a,g_k(\frac{1}{n+1})]|$, we get
\begin{equation}\label{eq: lewy dlugi prawy krotki }
d_n([a,b]) \geq \frac{m_n([a,\frac{1}{k}])}{|b-a|^{h_n}} \geq \frac{m_n([a,g_k(\frac{1}{n+1})])}{|b-a|^{h_n}} =
\end{equation}
\[
= \frac{\frac{m_n([a,g_k(\frac{1}{n+1})])}{|g_k(\frac{1}{n+1})- a|^{h_n}} \cdot |g_k(\frac{1}{n+1})- a|^{h_n}}{|b-a|^{h_n}} \geq \frac{|g_k(\frac{1}{n+1})- a|^{h_n} \cdot (1 - \frac{\varepsilon}{2})}{|b-a|^{h_n}} \geq
\]
\[
\geq \frac{|c - a|^{h_n} \cdot (1 - \frac{\varepsilon}{2})}{|b-a|^{h_n}}  = \frac{|I_1|^{h_n} \cdot (1 - \frac{\varepsilon}{2})}{|I_1+I_2|^{h_n}}  = \left( \frac{1}{2}\right)^{h_n} \cdot \left(1 - \frac{\varepsilon}{2} \right) \geq \frac{1}{2} - \varepsilon
\]
for sufficiently large n.
\par Case 3. For the second to last case, assume that $J_1$ is short (i.e. does not fulfill \eqref{eq: j1 long 2}) and $J_2$ is long (i.e. fulfills \eqref{eq: j2 long 2}). Fix $\varepsilon > 0$. Let $c$ be the center of the interval $[a,b]$. From Observation \ref{obs: j1 short then centered in j2} we know that $\frac{1}{k}<c$. Moreover, the interval $[\frac{1}{k},g_{k-1}\circ g_1(\frac{1}{n+1})]$ does not intersect $J_n$. Additionally, $[a,b]$ is centered at $J_n$ and thus $ g_{k-1}\circ g_1(\frac{1}{n+1}) \leq c$. Let $I_1 = [a,c]$ and $I_2 = [c,b]$. Because $c$ is the middle point of $[a,b]$, $|I_1| = |I_2|$. If $ \left |\frac{1}{k} - b \right| \geq \left | \frac{1}{k-1} - \frac{1}{k} \right | \cdot \frac{1}{2}$, then using Lemma \ref{lem: jednostajnosc bledu 2} we can find $n$ large enough such that 
\[
\frac{m_n\left(\left[d, 1 \right]\right)}{\left|1 - d \right |^{h_n}} \geq 1 - \frac{\varepsilon}{2}
\]
for all $d \leq \frac{1}{2}$. Since  $\left|\frac{1}{k}-b\right |\geq \left | \frac{1}{k-1}- \frac{1}{k} \right |\cdot \frac{1}{2} $, we can apply Lemma \ref{lem: niezmienniczosc gestosci} and obtain
\[
\frac{m_n\left(\left[g_{k-1} \circ g_1(\frac{1}{n+1}), b \right]\right)}{\left|b - g_{k-1} \circ g_1(\frac{1}{n+1}) \right |^{h_n}} \geq 1 - \frac{\varepsilon}{2}
\]
Applying this fact together with observation that $|[c,b]| \leq |[g_{k-1} \circ g_1 \left(\frac{1}{n+1} \right), b]|$ we get
\begin{equation}\label{eq: lewy krotki prawy dlugi}    
d_n([a,b]) \geq \frac{m_n([\frac{1}{k}, b])}{\left|a-b \right|^{h_n} } = \frac{m_n([g_{k-1}\circ g_1(\frac{1}{n+1}),b])}{|b-a|^{h_n}} = 
\end{equation}
\[
= \frac{m_n\left(\left[g_{k-1} \circ g_1(\frac{1}{n+1}), b \right]\right)}{\left|b - g_{k-1} \circ g_1(\frac{1}{n+1}) \right |^{h_n}} \cdot \frac{\left|b - g_{k-1} \circ g_1(\frac{1}{n+1}) \right|^{h_n}}{|b-a|^{h_n}} \geq \frac{(1 - \frac{\varepsilon}{2}) \cdot \left|b -c \right|^{h_n}}{|b-a|^{h_n}} \geq
\]
\[ 
\geq \frac{\left|I_2\right|^{h_n}\cdot\left(1 - \frac{\varepsilon}{2}\right)}{|I_1+I_2|^{h_n}} = \left(\frac{1}{2}\right)^{h_n} \cdot \left(1 - \frac{\varepsilon}{2} \right) \geq (\frac{1}{2} - \varepsilon)
\]
for sufficiently large $n$. 
\par Now if $\left |\frac{1}{k} - b \right| < \left | \frac{1}{k-1} - \frac{1}{k} \right | \cdot \frac{1}{2}$, then from Lemma \ref{lem: jednostajnosc bledu 1} find $n$ large enough to get 
\[
\frac{m_n\left(\left[\frac{1}{n+1}, d \right]\right)}{\left|d - \frac{1}{n+1} \right |^{h_n}} \geq 1 - \frac{\varepsilon}{2}
\]
for all $d \geq \frac{1}{n}$. Applying Lemma \ref{lem: niezmienniczosc gestosci} twice results in the following
\[
\frac{m_n\left(\left[g_{k-1} \circ g_1(\frac{1}{n+1}), b \right]\right)}{\left|b - g_{k-1} \circ g_1(\frac{1}{n+1}) \right |^{h_n}} \geq 1 - \frac{\varepsilon}{2}
\]
Using this result yields
\[
d_n([a,b]) \geq \frac{m_n([\frac{1}{k}, b])}{\left|a-b \right|^{h_n} } = \frac{m_n([g_{k-1}\circ g_1(\frac{1}{n+1}),b])}{|b-a|^{h_n}} \geq 
\]
\[
\geq \frac{m_n\left(\left[g_{k-1} \circ g_1(\frac{1}{n+1}), b \right]\right)}{\left|b - g_{k-1} \circ g_1(\frac{1}{n+1}) \right |^{h_n}} \cdot \frac{\left|b - g_{k-1} \circ g_1(\frac{1}{n+1}) \right|^{h_n}}{|b-a|^{h_n}} \geq \frac{(1 - \frac{\varepsilon}{2}) \cdot \left|b -c \right|^{h_n}}{|b-a|^{h_n}} \geq
\]
\[ 
\geq \frac{\left|I_2\right|^{h_n}\cdot\left(1 - \frac{\varepsilon}{2}\right)}{|I_1+I_2|^{h_n}} = \left(\frac{1}{2}\right)^{h_n} \cdot \left(1 - \frac{\varepsilon}{2} \right) \geq (\frac{1}{2} - \varepsilon)
\]
for sufficiently large $n$, which ends the proof for this case.
\\Case 4. For the last case, assume that both $J_1$ and $J_2$ are short.
We assumed that the interval $[a,b]$ must be centered at $J_n$. As before, denote by $c$ the center of the interval $[a,b]$. Because $J_1$ is short, we can invoke Observation \ref{obs: j1 short then centered in j2} and deduce that $c > \frac{1}{k}$. The interval $J_2$ is short as well and invoking Observation \ref{obs: j2 short then centered in j1} yields $c < \frac{1}{k}$. This implies that $[a,b]$ cannot be centered at $J_n$. 
\par The only thing left is to consider case when $k = n+1$. This case is even simpler, because the only possible case is Case 3. When $k = n+1$ then $\frac{1}{k} = \frac{1}{n+1}$ and thus $[a, \frac{1}{k}] \cap J_n = \emptyset$. Because the interval $[a,b]$ is centered at $J_n$, this implies that the center of the interval is in the interval $[\frac{1}{n+1},b]$. The proof for this situation is exactly the proof for the Case 3 of this proposition. 
\\
This ends the proof of Proposition \ref{prop: a do b z punktem w srodku}
\end{proof}
\subsection{Finalizing the proof}
As the last part of our proof, we put previously obtained results to get the following.
\begin{theorem}\label{thm: a do b}
\[
\liminf \limits_{n \to \infty} \inf \left \{d_n([a,b]): 0 < a < b \leq 1, \ [a,b] \text{ centered at } J_n \right \} \geq \frac{1}{2}
\]
\end{theorem}
\begin{proof}
If there does not exists $k \in \{1, 2, \dots, n\}$ such that $\frac{1}{k} \in (a,b)$, then we can apply Lemma \ref{lem: rozciaganie krotkich przedzialow} to the interval $[a,b]$. This yields interval $[\hat{a}, \hat{b}]$, with the same density, such that at least one $\frac{1}{k}$ belongs to $[\hat{a}, \hat{b}]$, $k \in \{1, 2, \dots, n\}$. Then by the results of Theorems \ref{thm: zero-r}, \ref{k+l do k}, \ref{krotkie lewe}, \ref{dlugie lewe}, \ref{krotkie prawe}, \ref{dlugie prawe} and Propositions \ref{prop: a do b z jednym w srodku}, \ref{prop: a do b z punktem w srodku} we get the thesis.
\end{proof}
This directly imples, using the definition of packing measure from \ref{defi: packing measure} and Theorem \ref{thm: tw o gestosci miary pakujacej}, that the upper limit of the packing measure is at most one, meaning
\begin{theorem}\label{thm: upper limit}
Let $S_n$ be IFS defined in \ref{IFS Sn}. Then
\[
\limsup \limits_{n \to \infty} \mathcal{P}_{h_n}(J_n) \leq 2 
\]
where $J_n$ is the limit set of the IFS $S_n$ and $P_h$ denotes packing measure in packing dimension $h$.
\end{theorem}
This result along with Theorem \ref{thm: ograniczenie z dolu} proves our main result.
\begin{theorem}\label{thm: granica}
Let $S_n$ be IFS defined in \ref{IFS Sn}. Then
\[
\lim \limits_{n \to \infty} \mathcal{P}_{h_n}(J_n) = 2
\]
where $J_n$ is the limit set of the iterated function system $S_n$ and $P_h$ denotes packing measure in packing dimension $h$.
\end{theorem}
\subsection{Proofs of the auxiliary lemmas}\label{subsec: dowody dodatkowych lematow}
Now, we will provide proof for auxiliary Lemma \ref{lem: jednostajnosc bledu 1} and \ref{lem: jednostajnosc bledu 2}.
\begin{lemma}\label{lem: jednostajnosc bledu 1}
\[
\liminf \limits_{n\to\infty} \left( \inf\limits_{d > \frac{1}{n}} \left\{\frac{m_n\left(\left[\frac{1}{n+1},d \right]\right)}{\left|d-\frac{1}{n+1} \right |^{h_n}} \right \} \right) \geq 1 
\]
\end{lemma}
\begin{proof}
Fix some $n > 0$. Take arbitrary $d>\frac{1}{n}$. We claim that in order to estimate the ratio
\begin{equation}\label{eq: star}
\frac{m_n\left(\left[\frac{1}{n+1},d \right]\right)}{\left|d-\frac{1}{n+1} \right |^{h_n}}
\end{equation}
from below, one can assume that $d \in J_n$, $d > \frac{1}{n}$. Indeed, if $d \not \in J_n$ (i.e. $d$ is in some "gap" of the Cantor set $J_n$), then we can replace $d$ by $\hat{d} := \inf \{x \in J_n: x > d \}$ (the right endpoint of the "gap"). Then $m_n([\frac{1}{n+1},d]) = m_n([\frac{1}{n+1}, \hat{d}])$ and $\left | d - \frac{1}{n+1}\right| \leq \left | \hat{d} - \frac{1}{n+1}\right| $  and, clearly, $\hat{d} \geq d > \frac{1}{n}$. So
\[
\frac{m_n\left(\left[\frac{1}{n+1},d \right]\right)}{\left|d-\frac{1}{n+1} \right |^{h_n}} \geq \frac{m_n\left(\left[\frac{1}{n+1},\hat{d} \right]\right)}{\left|\hat{d}-\frac{1}{n+1} \right |^{h_n}}
\]
So, from now on we assume that $d \in J_n$, $d > \frac{1}{n}$. 
\par Since $d \in J_n$, there exists a unique sequence of integers $(q_j)_{j=1}^\infty$, $q_j \leq n$ such that 
\[
\{ d \} = \bigcap \limits_{k=1}^\infty g_{q_1} \circ \dots \circ g_{q_k}([0,1])
\]
The measure $m_n([\frac{1}{n+1},d])$ can be expressed as the sum of the measures of the cylinder sets (i.e. intervals from the collection $\mathcal{F}^n_l$ located to the left of $d$). Summing up first the measures of the cylinder set of the first generation (see Figure \ref{fig: counted intervals}), we obtain 
\[
\sum \limits_{k = q_1+1}^n \left(\frac{1}{k} - \frac{1}{k+1} \right)^{h_n}
\]

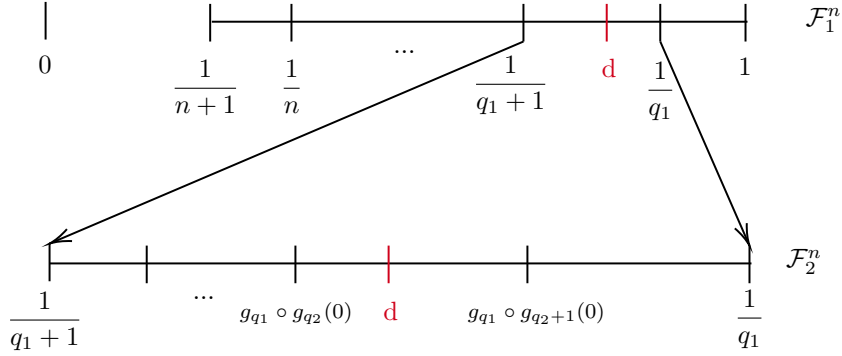
\begin{figure}
\tikzset{every picture/.style={line width=0.75pt}} %set default line width to 0.75pt        

\begin{tikzpicture}[x=0.75pt,y=0.75pt,yscale=-1,xscale=1]
%uncomment if require: \path (0,221); %set diagram left start at 0, and has height of 221

%Straight Lines [id:da48311335516161436] 
\draw    (242.83,26) -- (513,26) ;
%Straight Lines [id:da3909929361035801] 
\draw    (159,35) -- (159,16) ;
%Straight Lines [id:da4002647892883333] 
\draw    (512,36) -- (512,17) ;
%Straight Lines [id:da5267233304546214] 
\draw    (469,36) -- (469,17) ;
%Straight Lines [id:da5861927538557896] 
\draw    (400,36) -- (400,17) ;
%Straight Lines [id:da7580985350559125] 
\draw [color={rgb, 255:red, 208; green, 2; blue, 27 }  ,draw opacity=1 ]   (442,36) -- (442,17) ;
%Straight Lines [id:da4455370544073275] 
\draw    (283,36) -- (283,17) ;
%Straight Lines [id:da13071156713005105] 
\draw    (242,36) -- (242,17) ;
%Straight Lines [id:da979427429614078] 
\draw    (400,36) -- (162.84,137.21) ;
\draw [shift={(161,138)}, rotate = 336.89] [color={rgb, 255:red, 0; green, 0; blue, 0 }  ][line width=0.75]    (10.93,-3.29) .. controls (6.95,-1.4) and (3.31,-0.3) .. (0,0) .. controls (3.31,0.3) and (6.95,1.4) .. (10.93,3.29)   ;
%Straight Lines [id:da6947783911308332] 
\draw    (161,148) -- (515,148) ;
%Straight Lines [id:da4547627143904277] 
\draw    (161,157) -- (161,138) ;
%Straight Lines [id:da20302401726742514] 
\draw    (514,158) -- (514,139) ;
%Straight Lines [id:da024518973946355826] 
\draw    (402,158) -- (402,139) ;
%Straight Lines [id:da7070444697361384] 
\draw [color={rgb, 255:red, 208; green, 2; blue, 27 }  ,draw opacity=1 ]   (332,158) -- (332,139) ;
%Straight Lines [id:da6197744145899802] 
\draw    (285,158) -- (285,139) ;
%Straight Lines [id:da764346483349537] 
\draw    (210,158) -- (210,139) ;
%Straight Lines [id:da9592643755526419] 
\draw    (469,36) -- (513.2,137.17) ;
\draw [shift={(514,139)}, rotate = 246.4] [color={rgb, 255:red, 0; green, 0; blue, 0 }  ][line width=0.75]    (10.93,-3.29) .. controls (6.95,-1.4) and (3.31,-0.3) .. (0,0) .. controls (3.31,0.3) and (6.95,1.4) .. (10.93,3.29)   ;

% Text Node
\draw (154,42) node [anchor=north west][inner sep=0.75pt]   [align=left] {0};
% Text Node
\draw (507,43) node [anchor=north west][inner sep=0.75pt]   [align=left] {1};
% Text Node
\draw (438,43) node [anchor=north west][inner sep=0.75pt]  [color={rgb, 255:red, 208; green, 2; blue, 27 }  ,opacity=1 ] [align=left] {d};
% Text Node
\draw (333,39) node [anchor=north west][inner sep=0.75pt]   [align=left] {...};
% Text Node
\draw (221,44) node [anchor=north west][inner sep=0.75pt]   [align=left] {$\displaystyle \frac{1}{n+1}$};
% Text Node
\draw (328,166) node [anchor=north west][inner sep=0.75pt]  [color={rgb, 255:red, 208; green, 2; blue, 27 }  ,opacity=1 ] [align=left] {d};
% Text Node
\draw (231.78,163) node [anchor=north west][inner sep=0.75pt]   [align=left] {...};
% Text Node
\draw (255.46,166) node [anchor=north west][inner sep=0.75pt]  [font=\footnotesize] [align=left] {$\displaystyle g_{q_{1}} \circ g_{q_{2}}(0)$};
% Text Node
\draw (276,44) node [anchor=north west][inner sep=0.75pt]   [align=left] {$\displaystyle \frac{1}{n}$};
% Text Node
\draw (373,42) node [anchor=north west][inner sep=0.75pt]   [align=left] {$\displaystyle \frac{1}{q_{1}+1}$};
% Text Node
\draw (460,45) node [anchor=north west][inner sep=0.75pt]   [align=left] {$\displaystyle \frac{1}{q_{1}}$};
% Text Node
\draw (541,17.4) node [anchor=north west][inner sep=0.75pt]    {$\mathcal{F}_{1}^{n}$};
% Text Node
\draw (138,161) node [anchor=north west][inner sep=0.75pt]   [align=left] {$\displaystyle \frac{1}{q_{1}+1}$};
% Text Node
\draw (505,160) node [anchor=north west][inner sep=0.75pt]   [align=left] {$\displaystyle \frac{1}{q_{1}}$};
% Text Node
\draw (531,139.4) node [anchor=north west][inner sep=0.75pt]    {$\mathcal{F}_{2}^{n}$};
% Text Node
\draw (370.46,166) node [anchor=north west][inner sep=0.75pt]  [font=\footnotesize] [align=left] {$\displaystyle g_{q_{1}} \circ g_{q_{2} +1} (0)$};

\end{tikzpicture}
\caption{Length and measure of the intervals to the left of $d$.}\label{fig: counted intervals}
\end{figure}

Now, looking at the cylinder set containing $d$, (i.e. the interval $[\frac{1}{q_1+1}, \frac{1}{q_1}]$) we see that the cylinder sets of the second generations (i.e. the elements of the collection $\mathcal{F}^n_2$) contained in $[\frac{1}{q_1+1}, \frac{1}{q_1}]$ and located to the left of $d$ have length 
\[
\left( \frac{1}{q_1} -\frac{1}{q_1+1} \right) \cdot \left(\frac{1}{k} - \frac{1}{k+1} \right) 
\] $k = 1, \dots, q_2-1$, and measure 
\[
\left( \frac{1}{q_1} -\frac{1}{q_1+1} \right)^{h_n} \cdot \left(\frac{1}{k} - \frac{1}{k+1} \right)^{h_n}
\]
Proceeding by induction, we easily conclude that the measure $m_n\left(\left[\frac{1}{n+1},d \right]\right)$ can be expressed in the following form
\begin{equation}\label{eq: rownanie na mn}
m_n\left(\left[\frac{1}{n+1},d \right]\right) = \sum\limits_{k = q_1 + 1}^{n}  \left( \frac{1}{k} - \frac{1}{k+1} \right)^{h_n} + \left( \frac{1}{q_1} - \frac{1}{q_1 + 1} \right)^{h_n} \cdot \sum \limits_{k = 1}^{q_2 - 1} \left( \frac{1}{k} - \frac{1}{k+1} \right)^{h_n} + 
\end{equation}
\[
+ \prod\limits_{i = 1}^{2} (\frac{1}{q_i}-\frac{1}{q_i+1})^{h_n} \cdot \sum\limits_{k = q_3 + 1}^{n} \left( \frac{1}{k} - \frac{1}{k+1} \right)^{h_n} + \prod\limits_{i = 1}^{3} (\frac{1}{q_i}-\frac{1}{q_i+1})^{h_n} \cdot \sum\limits_{k = 1}^{q_4 - 1} \left( \frac{1}{k} - \frac{1}{k+1} \right)^{h_n} + \dots 
\]
%\[
%\leq \Bigg [\sum\limits_{k = q_1 + 1}^{n}  \left( \frac{1}{k} - \frac{1}{k+1} \right)+ \left( \frac{1}{q_1} - \frac{1}{q_1 + 1} \right) \cdot \sum \limits_{k = 1}^{q_2 - 1} \left( \frac{1}{k} - \frac{1}{k+1} \right) + 
%\]
%\[
%+ \prod\limits_{i = 1}^{2} (\frac{1}{q_i}-\frac{1}{q_i+1}) \cdot \left( \frac{1}{q_3+1} - \frac{1}{n+1} \right) + \prod\limits_{i = 1}^{3} (\frac{1}{q_i}-\frac{1}{q_i+1}) \cdot \sum\limits_{k = 1}^{q_4 - 1} \left( \frac{1}{k} - \frac{1}{k+1} \right) + \dots \Bigg]^{h_n}
%\]
%where the inequality comes from Lemma \ref{lem: wypuklosc z hn}.
Analogously, the value $|d - \frac{1}{n+1}|^{h_n}$ can be expressed as follows
\[
\left|d-\frac{1}{n+1} \right |^{h_n} =\Bigg [\sum\limits_{k = q_1 + 1}^{n}  \left( \frac{1}{k} - \frac{1}{k+1} \right) + \left( \frac{1}{q_1} - \frac{1}{q_1 + 1} \right)\cdot \sum \limits_{k = 1}^{q_2 - 1} \left( \frac{1}{k} - \frac{1}{k+1} \right) + 
\]
\[
+ \prod\limits_{i = 1}^{2} (\frac{1}{q_i}-\frac{1}{q_i+1})\cdot \sum\limits_{k = q_3 + 1}^{\infty} \left( \frac{1}{k} - \frac{1}{k+1} \right) + \prod\limits_{i = 1}^{3} \left(\frac{1}{q_i}-\frac{1}{q_i+1} \right)\cdot \sum\limits_{k = 1}^{q_4 - 1} \left( \frac{1}{k} - \frac{1}{k+1} \right) + \dots \Bigg]^{h_n}
\]
Now, clearly each summand
\[
\prod\limits_{i = 1}^{2k}\left(\frac{1}{q_i} - \frac{1}{q_i +1} \right) \cdot \sum\limits_{k = q_{2k+1}+1}^\infty \left(\frac{1}{k} - \frac{1}{k+1}\right)
\]
can be divided into two sums
\[
\prod\limits_{i = 1}^{2k}\left(\frac{1}{q_i} - \frac{1}{q_i +1} \right) \cdot \sum\limits_{k = q_{2k+1}+1}^n \left(\frac{1}{k} - \frac{1}{k+1}\right) + \prod\limits_{i = 1}^{2k}\left(\frac{1}{q_i} - \frac{1}{q_i +1} \right) \cdot \sum\limits_{k = n+1}^\infty \left(\frac{1}{k} - \frac{1}{k+1}\right)
\]
Note that the first sum corresponds to the sum appearing in the expression for $m_n([\frac{1}{n+1}, d])$. Grouping expressions that occur also in the formula for $m_n\left(\left[\frac{1}{n+1},d \right]\right) $ yields the following expression for $\left|d-\frac{1}{n+1} \right |^{h_n}$
\begin{equation}\label{eq: 1/n+1 do d dlugosc}
\left|d-\frac{1}{n+1} \right |^{h_n} = \Bigg[\sum\limits_{k = q_1 + 1}^{n}  \left( \frac{1}{k} - \frac{1}{k+1} \right) + \left( \frac{1}{q_1} - \frac{1}{q_1 + 1} \right) \cdot \sum \limits_{k = 1}^{q_2 - 1} \left( \frac{1}{k} - \frac{1}{k+1} \right) + 
\end{equation}
\[
+ \prod\limits_{i = 1}^{2} (\frac{1}{q_i}-\frac{1}{q_i+1}) \cdot \sum\limits_{k = q_3 + 1}^{n} \left( \frac{1}{k} - \frac{1}{k+1} \right) + \prod\limits_{i = 1}^{3} (\frac{1}{q_i}-\frac{1}{q_i+1}) \cdot \sum\limits_{k = 1}^{q_4 - 1} \left( \frac{1}{k} - \frac{1}{k+1} \right) + \dots +
\]
\[
+ \sum\limits_{i = 1}^\infty \prod \limits_{j = 1}^{2i} \left (\frac{1}{q_{j}} - \frac{1}{q_{j} + 1} \right) \cdot\sum\limits_{k = n +1}^{\infty} (\frac{1}{k} - \frac{1}{k+1}) \Bigg]^{h_n}
\]
Now, using Lemma \ref{lem: wypuklosc z hn} with $m_n\left(\left[\frac{1}{n+1},d \right]\right)$ and taking the quotient, we get
\[
\frac{\left|d-\frac{1}{n+1} \right |^{h_n}}{m_n\left(\left[\frac{1}{n+1},d \right]\right)} \leq
\]
\begin{adjustwidth}{-50pt}{0pt}
\[
\leq \left[1 + \frac{\sum\limits_{i = 1}^\infty \prod \limits_{j = 1}^{2i} \left (\frac{1}{q_{j}} - \frac{1}{q_{j} + 1} \right) \cdot\sum\limits_{k = n +1}^{\infty} (\frac{1}{k} - \frac{1}{k+1})}{\sum\limits_{k = q_1 + 1}^{n}  \left( \frac{1}{k} - \frac{1}{k+1} \right)+ \sum\limits_{i = 1}^\infty \prod \limits_{j = 1}^{2i} \left (\frac{1}{q_{j}} - \frac{1}{q_{j} + 1} \right) \cdot\left(\frac{1}{q_{2i+1}}- \frac{1}{n+1}\right)+ \sum\limits_{i = 1}^\infty \prod \limits_{j = 1}^{2i+1} \left (\frac{1}{q_{j}} - \frac{1}{q_{j} + 1} \right) \cdot\left(1- \frac{1}{q_{2i+2}}\right)} \right]^{h_n}  \leq
\]
\[
\leq \left[1 + \frac{\sum\limits_{i = 1}^\infty \prod \limits_{j = 1}^{2i} \left (\frac{1}{q_{j}} - \frac{1}{q_{j} + 1} \right) \cdot\left(\frac{1}{n+1}\right)}{\sum\limits_{k = q_1 + 1}^{n}  \left( \frac{1}{k} - \frac{1}{k+1} \right)+ \sum\limits_{i = 1}^\infty \prod \limits_{j = 1}^{2i} \left (\frac{1}{q_{j}} - \frac{1}{q_{j} + 1} \right) \cdot\left(\frac{1}{q_{2i+1}}- \frac{1}{n+1}\right)+ \sum\limits_{i = 1}^\infty \prod \limits_{j = 1}^{2i+1} \left (\frac{1}{q_{j}} - \frac{1}{q_{j} + 1} \right) \cdot\left(1- \frac{1}{q_{2i+2}}\right)} \right]^{h_n} \leq 
\]
\begin{equation}\label{eq: nierownosc z szeregiem}
\leq \left[1 + \frac{\left(\frac{1}{n+1}\right)\left(\frac{1}{q_1} - \frac{1}{q_1+1}\right)\left(\frac{1}{q_2} - \frac{1}{q_2+1}\right)\cdot\sum\limits_{i = 0}^\infty \left( \frac{1}{2}\right)^{i+1} }{\sum\limits_{k = q_1 + 1}^{n}  \left( \frac{1}{k} - \frac{1}{k+1} \right)+ \sum\limits_{i = 1}^\infty \prod \limits_{j = 1}^{2i} \left (\frac{1}{q_{j}} - \frac{1}{q_{j} + 1} \right) \cdot\left(\frac{1}{q_{2i+1}}- \frac{1}{n+1}\right)+ \sum\limits_{i = 1}^\infty \prod \limits_{j = 1}^{2i+1} \left (\frac{1}{q_{j}} - \frac{1}{q_{j} + 1} \right) \cdot\left(1- \frac{1}{q_{2i+2}}\right)} \right]^{h_n} \leq 
\end{equation}
\end{adjustwidth}
\[
\leq \left[1 + \frac{2}{n} \right]^{h_n}
\]
where the inequality \eqref{eq: nierownosc z szeregiem} comes form the fact that $\frac{1}{q_i} - \frac{1}{q_{i+1}} \leq \frac{1}{2}$ for every $i = 1, 2, \dots $. This implies that the numerator is limited from above by the following expression $\frac{1}{n} \cdot \left( \frac{1}{q_1} - \frac{1}{q_1+1} \right) \cdot \left( \frac{1}{q_2} - \frac{1}{q_2+1} \right) \cdot\sum \limits_{k = 1}^{\infty} \left( \frac{1}{2}\right)^k $. The last inequality follows from the fact, that $d > \frac{1}{q_1+1}$ and thus the denominator of the expression is limited from below by $\frac{1}{q_1+1} - \frac{1}{n+1}$.
From this, we get 
\[
\frac{m_n\left(\left[\frac{1}{n+1},d \right]\right)}{\left|d-\frac{1}{n+1} \right |^{h_n}} \geq \frac{1}{\left[1+\frac{2}{n}\right]^{h_n}} \geq \frac{1}{1+\frac{2}{n}} = \frac{n}{n+2} = 1 - \frac{2}{n+2}
\]
which concludes the proof.
\end{proof}
Very similar reasoning can be used to show the following. We provide the proof for completeness.
\begin{lemma}\label{lem: jednostajnosc bledu 2}
\[
\liminf \limits_{n\to\infty} \left( \inf\limits_{d < \frac{1}{2}} \left\{\frac{m_n\left(\left[d, 1 \right]\right)}{\left|1 - d \right |^{h_n}} \right \} \right) \geq 1 
\]
\end{lemma}
\begin{proof}
Fix some $n > 0$. Take arbitrary $d<\frac{1}{2}$. We claim that in order to estimate the ratio
\begin{equation}\label{eq: star2}
\frac{m_n\left(\left[d,1 \right]\right)}{\left|1-d\right |^{h_n}}
\end{equation}
from below, one can assume that $d \in J_n$, $d < \frac{1}{2}$. Indeed, if $d \not \in J_n$ (i.e. $d$ is in some "gap" of the Cantor set $J_n$), then we can replace $d$ by $\hat{d} := \inf \{x \in J_n: x < d \}$ (the left endpoint of the "gap"). Then $m_n([d,1]) = m_n([ \hat{d},1])$ and $\left | 1- d\right| \leq \left | 1 - \hat{d}\right| $  and, clearly, $\hat{d} \leq  d < \frac{1}{2}$. So
\[
\frac{m_n\left(\left[d,1 \right]\right)}{\left|1-d\right |^{h_n}} \geq \frac{m_n\left(\left[\hat{d},1 \right]\right)}{\left|1-\hat{d}\right |^{h_n}}
\]
So, from now on we assume that $d \in J_n$, $d < \frac{1}{2}$. 
\par Since $d \in J_n$, there exists a unique sequence of integers $(q_j)_{j=1}^\infty$, $q_j \leq n$ such that 
\[
\{ d \} = \bigcap \limits_{k=1}^\infty g_{q_1} \circ \dots \circ g_{q_k}([0,1])
\]
The measure $m_n([d,1])$ can be expressed as the sum of the measures of the cylinder sets (i.e. intervals from the collection $\mathcal{F}^n_l$ located to the right of $d$). Summing up first the measures of the cylinder set of the first generation (see Figure \ref{fig: intervals to the right}), we obtain 
\[
\sum \limits_{k = 1}^{q_1-1} \left(\frac{1}{k} - \frac{1}{k+1} \right)^{h_n}
\]
\begin{figure}
    \centering

\tikzset{every picture/.style={line width=0.75pt}} %set default line width to 0.75pt        

\begin{tikzpicture}[x=0.75pt,y=0.75pt,yscale=-1,xscale=1]
%uncomment if require: \path (0,221); %set diagram left start at 0, and has height of 221

%Straight Lines [id:da13165048266656776] 
\draw    (242.83,26) -- (513,26) ;
%Straight Lines [id:da40602589314108195] 
\draw    (159,35) -- (159,16) ;
%Straight Lines [id:da05585405846507374] 
\draw    (512,36) -- (512,17) ;
%Straight Lines [id:da7353593895361652] 
\draw    (393,36) -- (393,17) ;
%Straight Lines [id:da21698866784907433] 
\draw    (321,36) -- (321,17) ;
%Straight Lines [id:da8797481268301904] 
\draw [color={rgb, 255:red, 208; green, 2; blue, 27 }  ,draw opacity=1 ]   (363,35) -- (363,16) ;
%Straight Lines [id:da8319044725340906] 
\draw    (242,36) -- (242,17) ;
%Straight Lines [id:da6918356778759612] 
\draw    (321,36) -- (162.69,136.92) ;
\draw [shift={(161,138)}, rotate = 327.48] [color={rgb, 255:red, 0; green, 0; blue, 0 }  ][line width=0.75]    (10.93,-3.29) .. controls (6.95,-1.4) and (3.31,-0.3) .. (0,0) .. controls (3.31,0.3) and (6.95,1.4) .. (10.93,3.29)   ;
%Straight Lines [id:da34581546954576226] 
\draw    (161,148) -- (515,148) ;
%Straight Lines [id:da4247536663037478] 
\draw    (161,157) -- (161,138) ;
%Straight Lines [id:da5520394176754392] 
\draw    (514,158) -- (514,139) ;
%Straight Lines [id:da11918137440900267] 
\draw    (402,158) -- (402,139) ;
%Straight Lines [id:da28599256428472764] 
\draw [color={rgb, 255:red, 208; green, 2; blue, 27 }  ,draw opacity=1 ]   (353,158) -- (353,139) ;
%Straight Lines [id:da2648673651077942] 
\draw    (285,158) -- (285,139) ;
%Straight Lines [id:da4852913761514259] 
\draw    (210,158) -- (210,139) ;
%Straight Lines [id:da7232485590263065] 
\draw    (393,36) -- (512.48,137.7) ;
\draw [shift={(514,139)}, rotate = 220.41] [color={rgb, 255:red, 0; green, 0; blue, 0 }  ][line width=0.75]    (10.93,-3.29) .. controls (6.95,-1.4) and (3.31,-0.3) .. (0,0) .. controls (3.31,0.3) and (6.95,1.4) .. (10.93,3.29)   ;

% Text Node
\draw (154,42) node [anchor=north west][inner sep=0.75pt]   [align=left] {0};
% Text Node
\draw (507,43) node [anchor=north west][inner sep=0.75pt]   [align=left] {1};
% Text Node
\draw (359,42) node [anchor=north west][inner sep=0.75pt]  [color={rgb, 255:red, 208; green, 2; blue, 27 }  ,opacity=1 ] [align=left] {d};
% Text Node
\draw (449,51) node [anchor=north west][inner sep=0.75pt]   [align=left] {...};
% Text Node
\draw (221,44) node [anchor=north west][inner sep=0.75pt]   [align=left] {$\displaystyle \frac{1}{n+1}$};
% Text Node
\draw (348,166) node [anchor=north west][inner sep=0.75pt]  [color={rgb, 255:red, 208; green, 2; blue, 27 }  ,opacity=1 ] [align=left] {d};
% Text Node
\draw (231.78,163) node [anchor=north west][inner sep=0.75pt]   [align=left] {...};
% Text Node
\draw (255.46,166) node [anchor=north west][inner sep=0.75pt]  [font=\footnotesize] [align=left] {$\displaystyle g_{q_{1}} \circ g_{q_{2} \ }( 0)$};
% Text Node
\draw (299,47) node [anchor=north west][inner sep=0.75pt]   [align=left] {$\displaystyle \frac{1}{q_{1} +1}$};
% Text Node
\draw (383,49) node [anchor=north west][inner sep=0.75pt]   [align=left] {$\displaystyle \frac{1}{q_{1}}$};
% Text Node
\draw (541,17.4) node [anchor=north west][inner sep=0.75pt]    {$\mathcal{F}_{1}^{n}$};
% Text Node
\draw (138,161) node [anchor=north west][inner sep=0.75pt]   [align=left] {$\displaystyle \frac{1}{q_{1} +1}$};
% Text Node
\draw (505,160) node [anchor=north west][inner sep=0.75pt]   [align=left] {$\displaystyle \frac{1}{q_{1}}$};
% Text Node
\draw (531,139.4) node [anchor=north west][inner sep=0.75pt]    {$\mathcal{F}_{2}^{n}$};
% Text Node
\draw (370.46,166) node [anchor=north west][inner sep=0.75pt]  [font=\footnotesize] [align=left] {$\displaystyle g_{q_{1}} \circ g_{q_{2} +1}( 0)$};

\end{tikzpicture}

\caption{Length and measure of the intervals to the right of $d$.}
\label{fig: intervals to the right}
\end{figure}
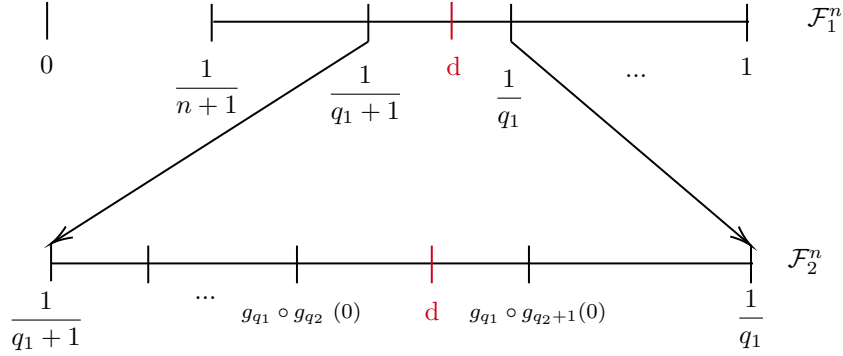
Now, looking at the cylinder set containing $d$, (i.e. the interval $[\frac{1}{q_1+1}, \frac{1}{q_1}]$) we see that the cylinder sets of the second generations (i.e. the elements of the collection $\mathcal{F}^n_2$) contained in $[\frac{1}{q_1+1}, \frac{1}{q_1}]$ and located to the right of $d$ have length 
\[
\left( \frac{1}{q_1} -\frac{1}{q_1+1} \right) \cdot \left(\frac{1}{k} - \frac{1}{k+1} \right) 
\] $k = q_2+1, \dots, n$, and measure 
\[
\left( \frac{1}{q_1} -\frac{1}{q_1+1} \right)^{h_n} \cdot \left(\frac{1}{k} - \frac{1}{k+1} \right)^{h_n}
\]
Proceeding by induction, we easily conclude that the measure $m_n\left(\left[d,1 \right]\right)$ can be expressed in the following form
\begin{equation}\label{eq: rownanie na mn2}
m_n([1 - d]) = \sum\limits_{k = 1}^{q_1 - 1}  \left( \frac{1}{k} - \frac{1}{k+1} \right)^{h_n} + \left( \frac{1}{q_1} - \frac{1}{q_1 + 1} \right)^{h_n} \cdot \sum \limits_{k = q_2 + 1 }^{n} \left( \frac{1}{k} - \frac{1}{k+1} \right)^{h_n} +
\end{equation}
\[
+ \prod\limits_{i = 1}^{2} \left(\frac{1}{q_i}-\frac{1}{q_i+1}\right)^{h_n} \cdot \sum\limits_{k = 1}^{q_3 - 1} \left( \frac{1}{k} - \frac{1}{k+1} \right)^{h_n} + \prod\limits_{i = 1}^{3} \left(\frac{1}{q_i}-\frac{1}{q_i+1}\right)^{h_n} \cdot \sum\limits_{k = q_4 + 1}^{n} \left( \frac{1}{k} - \frac{1}{k+1} \right)^{h_n} + \dots
\]
Analogously, the value $|1-d|^{h_n}$ can be expressed as follows
\[
\left|1 - d \right |^{h_n} = \Bigg[\sum\limits_{k = 1}^{q_1 - 1}  \left( \frac{1}{k} - \frac{1}{k+1} \right) + \left( \frac{1}{q_1} - \frac{1}{q_1 + 1} \right) \cdot \sum \limits_{k = q_2 + 1 }^{\infty} \left( \frac{1}{k} - \frac{1}{k+1} \right) + 
\]
\[
+ \prod\limits_{i = 1}^{2} \left(\frac{1}{q_i}-\frac{1}{q_i+1}\right) \cdot \sum\limits_{k = 1}^{q_3 - 1} \left( \frac{1}{k} - \frac{1}{k+1} \right) + \prod\limits_{i = 1}^{3} \left(\frac{1}{q_i}-\frac{1}{q_i+1}\right) \cdot \sum\limits_{k = q_4 + 1}^{\infty} \left( \frac{1}{k} - \frac{1}{k+1} \right) + \dots\Bigg]^{h_n}
\]
Now, clearly each summand
\[
\prod\limits_{i = 1}^{2k+1}\left(\frac{1}{q_i} - \frac{1}{q_i +1} \right) \cdot \sum\limits_{k = q_{2k+2}+1}^\infty \left(\frac{1}{k} - \frac{1}{k+1}\right)
\]
can be divided into two sums
\[
\prod\limits_{i = 1}^{2k+1}\left(\frac{1}{q_i} - \frac{1}{q_i +1} \right) \cdot \sum\limits_{k = q_{2k+2}+1}^n \left(\frac{1}{k} - \frac{1}{k+1}\right) + \prod\limits_{i = 1}^{2k+1}\left(\frac{1}{q_i} - \frac{1}{q_i +1} \right) \cdot \sum\limits_{k = n+1}^\infty \left(\frac{1}{k} - \frac{1}{k+1}\right)
\]
Note that the first sum corresponds to the sum appearing in the expression for $m_n([d, 1])$. Grouping expressions that occur also in the formula for $m_n\left(\left[d, 1 \right]\right) $ yields the following expression for $\left|1-d \right |^{h_n}$
\begin{equation}\label{eq: 1/n+1 do d dlugosc2}
\left|1 - d \right |^{h_n} = \Bigg[\sum\limits_{k = 1}^{q_1 - 1}  \left( \frac{1}{k} - \frac{1}{k+1} \right) + \left( \frac{1}{q_1} - \frac{1}{q_1 + 1} \right) \cdot \sum \limits_{k = q_2 + 1 }^{n} \left( \frac{1}{k} - \frac{1}{k+1} \right) + 
\end{equation}
\[
+ \prod\limits_{i = 1}^{2} \left(\frac{1}{q_i}-\frac{1}{q_i+1}\right) \cdot \sum\limits_{k = 1}^{q_3 - 1} \left( \frac{1}{k} - \frac{1}{k+1} \right) + \prod\limits_{i = 1}^{3} \left(\frac{1}{q_i}-\frac{1}{q_i+1}\right) \cdot \sum\limits_{k = q_4 + 1}^{n} \left( \frac{1}{k} - \frac{1}{k+1} \right) + \dots
\]
\[ 
+ \sum\limits_{i = 1}^\infty \prod \limits_{j = 1}^{2i+1} \left (\frac{1}{q_{j}} - \frac{1}{q_{j} + 1} \right) \cdot\sum\limits_{k = n +1}^{\infty} (\frac{1}{k} - \frac{1}{k+1}) \Bigg]^{h_n}
\]
Now, using Lemma \ref{lem: wypuklosc z hn} with $m_n\left(\left[d, 1 \right]\right)$ and taking the quotient, we get
\[
\frac{\left|1 - d \right |^{h_n}}{m_n([1 - d])} \leq
\]
\begin{adjustwidth}{-50pt}{0pt}
\[
\leq \left[1 + \frac{\sum\limits_{i = 1}^\infty \prod \limits_{j = 1}^{2i+1} \left (\frac{1}{q_{j}} - \frac{1}{q_{j} + 1} \right) \cdot\sum\limits_{k = n +1}^{\infty} (\frac{1}{k} - \frac{1}{k+1}) }{\sum\limits_{k = 1}^{q_1-1}  \left( \frac{1}{k} - \frac{1}{k+1} \right) + \sum\limits_{i = 0}^\infty \prod \limits_{j = 1}^{2i+1} \left (\frac{1}{q_{j}} - \frac{1}{q_{j} + 1} \right)\cdot\left(\frac{1}{q_{2i+2}}- \frac{1}{n+1}\right) + \sum\limits_{i = 1}^\infty \prod \limits_{j = 1}^{2i} \left (\frac{1}{q_{j}} - \frac{1}{q_{j} + 1} \right)\cdot\left(1- \frac{1}{q_{2i+1}}\right)} \right]^{h_n}  =
\]
\[
=\left[1 + \frac{ \frac{1}{n+1}\cdot \left( \frac{1}{q_1} - \frac{1}{q_1 + 1} \right) \cdot \sum\limits_{i = 1}^\infty \prod \limits_{j = 1}^{i} \left (\frac{1}{q_{2j}} - \frac{1}{q_{2j} + 1} \right) \cdot \left(\frac{1}{q_{2j+1}} - \frac{1}{q_{2j+1} + 1} \right)}{\sum\limits_{k = 1}^{q_1-1}  \left( \frac{1}{k} - \frac{1}{k+1} \right) + \sum\limits_{i = 0}^\infty \prod \limits_{j = 1}^{2i+1} \left (\frac{1}{q_{j}} - \frac{1}{q_{j} + 1} \right)\cdot\left(\frac{1}{q_{2i+2}}- \frac{1}{n+1}\right) + \sum\limits_{i = 1}^\infty \prod \limits_{j = 1}^{2i} \left (\frac{1}{q_{j}} - \frac{1}{q_{j} + 1} \right)\cdot\left(1- \frac{1}{q_{2i+1}}\right)} \right]^{h_n}  \leq
\]
\begin{equation}\label{eq: nierownosc z 1/2}
\leq \left[1 + \frac{ \frac{1}{n+1}\cdot \left( \frac{1}{q_1} - \frac{1}{q_1 + 1} \right) \cdot \sum\limits_{i = 1}^\infty \left(\frac{1}{2}\right)^{i}}{\sum\limits_{k = 1}^{q_1-1}  \left( \frac{1}{k} - \frac{1}{k+1} \right) + \sum\limits_{i = 0}^\infty \prod \limits_{j = 1}^{2i+1} \left (\frac{1}{q_{j}} - \frac{1}{q_{j} + 1} \right)\cdot\left(\frac{1}{q_{2i+2}}- \frac{1}{n+1}\right) + \sum\limits_{i = 1}^\infty \prod \limits_{j = 1}^{2i} \left (\frac{1}{q_{j}} - \frac{1}{q_{j} + 1} \right)\cdot\left(1- \frac{1}{q_{2i+1}}\right)} \right]^{h_n}    \leq 
\end{equation}
\end{adjustwidth}
\[
\left[1 + \frac{2}{n} \right]^{h_n}
\]
where the inequality \eqref{eq: nierownosc z 1/2} comes form the fact that $\frac{1}{q_i} - \frac{1}{q_{i+1}} \leq \frac{1}{2}$ for every $i = 1, 2, \dots $. This implies that the numerator is limited from above by the expression $\frac{1}{n} \cdot \left( \frac{1}{q_1} - \frac{1}{q_1+1} \right)\cdot\sum \limits_{k = 1}^{\infty} \left( \frac{1}{2}\right)^k $. The last inequality follows from the fact, that $d \leq \frac{1}{q_1}$. From this, we get 
\[
\frac{m_n([1 - d])}{\left|1 - d \right |^{h_n}} \geq \frac{1}{\left[1+\frac{2}{n}\right]^{h_n}} \geq \frac{1}{1+\frac{2}{n}} = \frac{n}{n+2} = 1 - \frac{2}{n+2}
\]
which concludes the proof.

\end{proof}
\bibliography{main}

\begin{thebibliography}{10}

\bibitem{AyerStri}
Elizabeth Ayer and Robert~S. Strichartz.
\newblock Exact hausdorff measure and intervals of maximum density for cantor sets.
\newblock {\em Transactions of the American Mathematical Society}, 351:3725--3741, 1999.

\bibitem{Falconerdrugi}
Kenneth Falconer.
\newblock {\em Techniques in Fractal Geometry}.
\newblock John Wiley \& Sons, 1997.

\bibitem{Falconer}
Kenneth Falconer.
\newblock {\em Fractal Geometry: Mathematical Foundations and Applications, 2nd Edition}.
\newblock John Wiley \& Sons, 2004.

\bibitem{Feng}
De–Jun Feng.
\newblock Exact packing measure of linear cantor sets.
\newblock {\em Mathematische Nachrichten}, 248-249:102--109, 2003.

\bibitem{Feng2}
De–Jun Feng.
\newblock Exact packing measure of linear cantor sets.
\newblock {\em Mathematische Nachrichten}, 248-249:102--109, 2003.

\bibitem{Hensley}
Doug Hensley.
\newblock Continued fraction cantor sets, hausdorff dimension, and functional analysis.
\newblock {\em Journal of Number Theory}, 40:336--358, 1992.

\bibitem{jarnik1}
Vojtĕch Jarník.
\newblock Zur metrischen theorie der diophantischen approximationen.
\newblock {\em Prace Matematyczno-Fizyczne}, 36:91--106, 1928-1929.

\bibitem{Matilla}
Pertti Mattila.
\newblock {\em Geometry of Sets and Measures in Euclidean Spaces: Fractals and Rectifiability}.
\newblock Cambridge University Press, 1995.

\bibitem{przyturb}
Feliks Przytycki and Mariusz Urbański.
\newblock {\em Conformal Fractals: Ergodic Theory Methods}.
\newblock Cambridge University Press, 2010.

\bibitem{UZ}
Mariusz Urba\'nski and Anna Zdunik.
\newblock Continuity of the {Hausdorff} {Measure} of {Continued} {Fractions} and {Countable} {Alphabet} {Iterated} {Function} {Systems}.
\newblock {\em Journal de th\'eorie des nombres de Bordeaux}, 28:261--286, 2016.

\end{thebibliography}

\end{document}